\documentclass[11pt, oneside]{article} 

\usepackage[utf8]{inputenc}
\usepackage[T1]{fontenc}
\usepackage{xcolor} 
\usepackage{color} 

\usepackage{amsmath}
\usepackage{amssymb} 
\usepackage{mathrsfs} 
\usepackage{bm} 
\usepackage{bbm} 
\usepackage{float}
\usepackage{multirow}
\usepackage{comment}

\usepackage{nameref}
\usepackage{xr-hyper}
\usepackage{makecell, float}
\usepackage{amsthm} 
\usepackage{amsmath}
\theoremstyle{plain}
\newtheorem{theorem}{Theorem}[section] 

\newtheorem{lemma}[theorem]{Lemma}

\newtheorem{remark}[theorem]{Remark}

\usepackage{algorithm2e} 
\SetKwInput{KwInputs}{Inputs}
\usepackage[noend]{algpseudocode}
\makeatletter
\def\BState{\State\hskip-\ALG@thistlm}
\makeatother

\usepackage{graphicx} 
\usepackage{subcaption} 
\graphicspath{{figures/}}

\usepackage{tikz} 
\usetikzlibrary{positioning ,shadows,matrix}
\usepackage{pgfplots}
\pgfplotsset{compat=1.18}
\usepgfplotslibrary{groupplots}
\usepackage[english]{babel} 

\usepackage[]{natbib}
\usepackage{bookmark}

\IfFileExists{xurl.sty}{\usepackage{xurl}}{} 
\hypersetup{
	pdftitle={Title},
	pdfauthor={Author 1; Author 2},
	pdfkeywords={3 to 6 keywords, that do not appear in the title},
	colorlinks=true,
	linkcolor={blue},
	filecolor={Maroon},
	citecolor={Blue},
	urlcolor={Blue},
	pdfcreator={LaTeX via pandoc}}

\usepackage{hyperref} 
\hypersetup{
	pdfpagemode={UseOutlines},
	bookmarksopen,
	pdfstartview={FitH},
	colorlinks,
	linkcolor={blue}, 
	citecolor={blue}, 
	urlcolor={blue} 
}

\usepackage[letterpaper,margin=1.25in]{geometry} 

\usepackage{enumitem} 

\def\iid{\overset{\textnormal{iid}}{\sim}} 
\DeclareMathOperator{\sgn}{sgn}

\makeatletter  
\@ifundefined{@currsizeindex} 
{\RequirePackage{relsize}\let\dolarger\relsize} 
{\def\dolarger#1{\larger[#1]}} 
\newcommand*\@@bigtimes[2]{\vphantom{\prod} 
	\vcenter{\hbox{\dolarger{4}$\m@th#1\mkern-2mu\times\mkern-2mu$}}} 
\newcommand*\bigtimes{\mathop{\mathpalette\@@bigtimes\relax}\displaylimits} 
\makeatother

\def\iid{\overset{\textnormal{iid}}{\sim}} 
\def\cube{[0,1]^d} 

\def\N{\mathbb{N}}\def\R{\mathbb{R}}\def\1{\mathbbm{1}}

\def\Pcal{\mathcal{P}}\def\Wcal{\mathcal{W}}

\title{\bf Hierarchical Besov-Laplace priors for spatially inhomogeneous binary classification}
\author{Patric Dolmeta\\
	ESOMAS Department, University of Turin\\
	and \\
	Matteo Giordano\thanks{
		M.G. has been partially supported by MUR, PRIN project 2022CLTYP4. The authors gratefully acknowledge the support from the ``de Castro" Statistics Initiative.}\\
	ESOMAS Department, University of Turin}
\date{} 

\begin{document}
	
	\maketitle
	
	\begin{abstract}
	We study nonparametric Bayesian binary classification, in the case where the unknown probability response function is possibly spatially inhomogeneous, for example, being generally flat across the domain but presenting localized sharp variations. We consider a hierarchical procedure based on the Besov-Laplace priors from the inverse problems and imaging literature, with a carefully tuned hyper-prior on the regularity parameter. We show that the resulting posterior distribution concentrates towards the ground truth at optimal rate, automatically adapting to the unknown regularity. To implement posterior inference in practice, we devise an efficient Markov chain Monte Carlo (MCMC) algorithm based on recent ad-hoc dimension-robust methods for Besov-Laplace priors. We then test the considered approach in extensive numerical simulations, where we obtain a solid corroboration of the theoretical results.
	\end{abstract}

	\vspace{24pt}
	
	
	%
	
	\tableofcontents

\section{Introduction}\label{Sec:Intro}

Consider the binary classification problem: To predict a $0$-$1$ response $Y$ from the value of a possibly multi-dimensional covariate $X$. This task is canonically approached by modeling $Y$, conditionally given $X$, as a Bernoulli random variable with success probability $p(X)$. Given labeled data $D^{(n)}:=\{(Y_i,X_i)\}_{i=1}^n$, the goal is then to obtain an estimate $\hat p$ of the `probability response function' $x\mapsto p(x) = \Pr(Y = 1 | X = x)$. Using the latter, new unlabeled inputs $X_{n+1},X_{n+2},\dots$, can be classified based on the plug-in estimates $\hat p(X_{n+1}),\hat p(X_{n+2}),\dots$; for example, to belong to class $1$ if the corresponding probabilities are greater than a certain threshold.

In this article, we consider the nonparametric Bayesian approach to the binary classification problem. This entails modeling $p$ with a suitable prior probability measure $\Pi$ on a function space, and then combining $\Pi$ with the data, via Bayes' theorem, to form the posterior distribution $\Pi(\cdot|D^{(n)})$, that is the conditional law of $p|D^{(n)}$. Following the Bayesian paradigm, $\Pi(\cdot|D^{(n)})$ represents the updated belief about $p$ after observing $D^{(n)}$, furnishing point estimates as well as uncertainty quantification. See Section \ref{Sec:Model} for details, and \cite[Chapter 1]{GvdV17} for a general overview of the methodology. The arguably most widely adopted prior distributions for classification surfaces are the ones based on Gaussian processes, for which there exists an extensive literature providing methodological and computational strategies, e.g.~\cite{lenk1988logistic,choudhuri2007nonparametric,nickisch2008approximations}, as well as theoretical performance guarantees, e.g.~\cite{ghosal2006posterior,vdVvZ08,vdVvZ09}. Further, see \cite[Chapter 3]{RW06}. Other commonly used approaches include procedures with Dirichlet process priors, Dirichlet process mixture models, and mixture of experts; see \cite{gelfand1991nonparametric,jara2007dirichlet,wang2010classification}, where many additional references can be found.

In many applications, it is natural to expect that the target probability response function exhibits both general trends as well as localized features, whose correct detection is central to the efficacy of the employed classification procedure. For instance, there may be unknown `critical values' of the covariates that induce sharp variations in the probability of success, resulting in localised high variation or even discontinuities. Notable examples arise in pharmacology or toxicology where the minimum effective dose and the threshold of toxicological concern are well documented critical values affecting the probability of a therapeutic or toxic event, \cite{Li17_tox}. Similarly, the point of material fatigue, the fracture point of elastic materials and their yield point all represent sharp transitional values determining the probability of structural failure monitored by engineers, \cite{alava06}.

To model this scenario, we employ Besov-Laplace priors. These are widely used in inverse problems and imaging, \cite{LSS09}, because of their excellent empirical performances in recovering spatially inhomogeneous objects, as well as their ability to induce sparse reconstructions and provide edge detection at the level of the maximum-a-posteriori (MAP) estimator. See e.g.~\cite{LP01,BD06,VLS09,SE15,KLSS23} and references therein, and Figure \ref{Fig:2D_bs} for an illustration. Throughout, we will refer to `edge-preserving' and `spatial inhomogeneity' in accordance with the common terminology in the aforementioned literature. However, we stress that in the present classification setting, the reference domain is the covariate space, and these effects refer to sharp variations in the probability of success, as discussed in the examples above. 

Besov-Laplace priors are defined via wavelet series with independent random coefficients following rescaled Laplace distributions, furnishing an infinite-dimensional version of the total-variation prior of \cite{ROF92}, while also maintaining a favorable log-concave structure that enhances computation, \cite{BG15}, and analytical study, \cite{ADH21}. Their construction allows the wavelet coefficients to be sparse and so to `activate' high-frequencies only at the location of the critical values, creating `locally sharp' functions without affecting smoothness elsewhere. In contrast, Gaussian priors are suited to model functions with milder variability, and have been shown to be unable to optimally reconstruct more structured signals, \cite{ADH21,giordano2022inability,agapiou2024laplace}. See Section \ref{Sec:Simulations} for an illustration with synthetic data, where the higher empirical reconstruction quality achieved by Besov-Laplace priors over Gaussian priors in the presence of spatially inhomogeneous ground truths can be visualised.

The study of the large sample properties of posterior distributions based on Besov-Laplace priors has been recently initiated in \cite{ADH21}, following the landmark developments in the theory of the frequentist analysis of nonparametric Bayesian procedures over the last two decades, \cite{GGvdV00,shen2001rates,GvdV07,vdVvZ08,GN11}. Among their results, they showed that, in the white noise model, properly tuned Besov-Laplace priors achieve minimax-optimal posterior contraction rates towards ground truths $p_0$ in the Besov scale $B^\alpha_{1}$, $\alpha>0$. These function spaces are defined via wavelet series expansions with $\ell^1$-type penalties on the wavelet coefficients, measuring local variations in an $L^1$-sense and allowing for spatial inhomogeneity. See Section \ref{Subsec:Notation} for definitions, and \cite[Section 1]{DJ98} for a detailed description of the connection to the space of bounded variation functions. These results were later extended to various statistical models, including drift estimation for diffusion processes, \cite{GR22}, density estimation, \cite{giordano23besov}, and nonlinear inverse problems, \cite{agapiou2024laplace}. We further refer the reader to the recent investigation by \cite{dolera2024strong}, as well as to earlier related results by \cite{CN14,R13,AGR13}.

In the context of nonparametric binary classification, the recent work by \cite{giordano2025bayesian} showed that Besov-Laplace priors can yield optimal reconstruction of spatially inhomogeneous probability response functions. However, a limitation of their result is the lack of `adaptation', that is, the (often unrealistic) requirement of knowing the regularity of the ground truth in order to correctly tune the procedure to achieve the optimal rate. See \cite[Chapter 10]{GvdV17} for a general overview of the problem of adaptation in Bayesian nonparametrics. To our knowledge, for Besov-Laplace priors, this issue has so far been investigated only by \cite{giordano23besov} in density estimation, using the hierarchical Bayesian approach, and by \cite{agapiou2024adaptive} in the white noise model for both hierarchical and empirical methods.

Here, we build on the latter studies, and consider a hierarchical prior for classification surfaces obtained by assigning a hyper-prior to the regularity hyper-parameter within a base rescaled Besov-Laplace prior (combined with a suitable link function). We show that the resulting posterior distribution achieves optimal posterior contraction rates towards any $p_0\in B^\alpha_{1}$, simultaneously for all $\alpha$ greater than a minimal threshold, without requiring knowledge of $\alpha$ and thereby adapting to the smoothness of $p_0$, cf.~Theorem \ref{Theo:LaplRates}. The proof is based on the general approach to posterior contraction rates in total variation distance, \cite{GGvdV00}, which we pursue by carefully constructing the hyper-prior for the smoothness. See Appendix \ref{Sec:Proof}.

A secondary contribution of this article is an investigation of the implementation aspects of Besov-Laplace priors for binary classification. Since in the setting at hand the posterior distribution is not available in closed form, we devise an efficient sampling
scheme employing recent ad-hoc dimension robust Markov chain Monte Carlo (MCMC) techniques, \cite{CDPS18}. We then test the considered methodology in several simulation studies, providing a solid corroboration of the theory, cf.~Section \ref{Sec:Simulations}. In the experiments, we consider both spatially homogeneous and inhomogeneous ground truths, in one- and bi-dimensional scenarios.

The remainder of the paper is organized as follows. Section \ref{Subsec:Notation} summarizes basic definitions and the main notation used throughout. The hierarchical Besov-Laplace prior for probability response functions is constructed in Section \ref{Subsec:Prior}. The main asymptotic result is provided in Section \ref{Subsec:LaplRates}. The employed MCMC scheme is outlined in Section \ref{Subsec:Algorithm}. In Section \ref{Sec:Simulations}, we present the simulation studies. A summary of results and outlook on related research questions is included in Section \ref{Sec:Discussion}. The proofs are developed in Appendix \ref{Sec:Proof}.

%
%
%

\subsection{Main notation}
\label{Subsec:Notation}

In the following, we take the $d$-dimensional unit cube $\cube, \ d\in\N$, as the primary working domain. For $p\in[1,\infty]$, let $L^p(\cube)$ be the usual Lebesgue spaces of $p$-integrable functions defined on $\cube$, and let $\|\cdot\|_p$ be their norms. Write $\langle\cdot,\cdot\rangle_2$ for the inner product of $L^2(\cube)$.

Let $(\psi_{l}, \ l\in\N)$ be an orthonormal wavelet basis of $L^2(\cube)$, ordered with a single index, comprising $S$-regular, $S\in\N$, compactly supported and boundary corrected Daubechies wavelets; see \cite[Appendix A]{LSS09} for definitions and details. For $\alpha\in[0,S)$ and $p\in[1,\infty)$, define the (wavelet-based) Besov spaces
\begin{align*}
	B^\alpha_{p}(\cube)
	:= \left\{w = \sum_{l=1}^\infty w_l \psi_l :
    \|w\|^p_{B^\alpha_p} := \sum_{l=1}^\infty l^{p(\alpha/d+1/2)-1}|w_l|^p<\infty\right\},
\end{align*}
cf.~\cite[Appendix A]{LSS09}. For $p=\infty$, the spaces $B^\alpha_\infty(\cube)$, $\alpha\ge0$, are defined as above by replacing the $\ell^p$-type norm with the corresponding $\ell^\infty$-type one. The (fixed) regularity $S$ of the underlying wavelet basis can be taken arbitrarily large; thus, we tacitly imply that the condition $\alpha < S$ be satisfied throughout. The traditional Hilbert-Sobolev spaces $H^\alpha(\cube)$ and H\"older spaces $C^\alpha(\cube)$ are known to be contained within the above family. Specifically, for all $\alpha\ge0$, $B^\alpha_{2}(\cube) = H^\alpha(\cube)$, e.g.~\cite[p.~370]{GN16}, and $C^\alpha(\cube)\subseteq B^\alpha_{\infty}(\cube)$, with equality holding when $\alpha\notin\N$, e.g.~\cite[p.~370]{GN16}. On the other hand, for $p=1$, the $B^\alpha_1$-Besov scale is known to suitably model spatially inhomogeneous functions with localized `spiky' or `blocky' features. For example, the space of bounded variation functions $BV(\cube)$, which is of particular interest in many applications, is closely related to $B^1_{1}(\cube)$; see \cite[Section 1]{DJ98}.

When no confusion can arise, we at times omit the underlying domain in the notation, writing, for example, $B^\alpha_{p}$ for $B^\alpha_{p}(\cube)$. We use the symbols $\lesssim,\ \gtrsim$, and $\simeq$ for one- and two-sided inequalities holding up to multiplicative constants that are independent of all the involved quantities. For a set $\Wcal$ and a metric $\delta$ on $\Wcal$, the covering number $N(\varepsilon;\Wcal,\delta)$, $\varepsilon>0$, is defined as the minimum number of balls of $\delta$-radius equal to $\varepsilon$ needed to contain $\Wcal$ in their union. 

%
%
%
%
%

\section{Hierarchical Besov-Laplace priors for binary classification}\label{Sec:Model}

Let $D^{(n)}=\{ (Y_i,X_i)\}_{i = 1}^n$ be binary labeled classification data with random design, generated according to the statistical model
\begin{equation}
\label{Eq:Model}
\begin{split}
    Y_i | X_i & \iid \text{Bernoulli}(p(X_i)), \\
    X_i & \iid \mu_X,
\end{split}
\end{equation}
where $\mu_X$ is an absolutely continuous probability distribution on a compact subset of $\R^d$, which we take to be $[0,1]^d$ throughout for convenience. Other compact domains can be treated via appropriate scalings and minor notational changes, while extensions to unbounded covariates and to discrete ones are discussed in Section \ref{Sec:Discussion}. In slight abuse of notation, we also write $\mu_X$ for the probability density function, p.d.f., of the covariates. Above, $p:[0,1]^d \to [0,1]$ is an unknown probability response function, which we also call the `classification surface'. We consider the problem of nonparametrically estimating $p$ from observations $D^{(n)}$. We denote by $Q_p^{(n)}$ the joint (product) law of $D^{(n)}$, and by $E^{(n)}_p$ the expectation with respect to it. The likelihood is given by
\begin{equation}
    \label{Eq:Likelihood}
	L^{(n)}(p) =  \prod_{i = 1}^n p(X_i)^{Y_i} (1 - p(X_i))^{1-Y_i}\mu_X(X_i).
\end{equation}

As $L^{(n)}$ depends on $\mu_X$ only through multiplicative terms that are independent of $p$, knowledge of $\mu_X$ is not required for inference on $p$, cf.~after eq.~\eqref{Eq:Post}. For our theoretical results, we will only impose the minimal requirement that the p.d.f.~$\mu_X$ be bounded and bounded away from zero. In the scenario where estimating $\mu_X$ is also of interest, standard density estimation techniques, e.g.~\cite{T09,GN16,GvdV17} could be used based on the marginal sample $X_1,\dots,X_n$, without impacting likelihood-based inference procedures for $p$.

%

%
%
%

\subsection{The prior model}
\label{Subsec:Prior}

We adopt the nonparametric Bayesian approach, cf.~\cite{GvdV17}, modeling $p$ with a prior distribution $\Pi$ supported on the collection $\Pcal$ of measurable functions defined on $[0,1]^d$ and with values in $[0,1]$. In particular, we address the case where $p$ may be spatially inhomogeneous, possibly presenting localized features such as spikes or sharp variations. Since low-integrability Besov spaces provide an effective mathematical model for functions of this type, we make the assumption that $p$ belongs to the $B^\alpha_1$-Besov scale; see Section \ref{Subsec:Notation} for definitions and details. We then employ Besov-Laplace priors, which constitutes a particular case within the general class of `Besov space priors' introduced by \cite{LSS09}. These are commonly used in inverse problems and imaging due to their `edge-preserving and sparsity promoting' properties, e.g.~\cite{LP01,VLS09,KLSS23}, and have been recently shown to yield optimal recovery, in various statistical models, of spatially inhomogeneous functions in Besov spaces; see \cite{ADH21} as well as \cite{GR22,giordano23besov,agapiou2024laplace}.

Building on the latter references, we define priors for classification surfaces starting from rescaled Besov-Laplace random functions
\begin{equation}
\label{Eq:BaseLaplPrior}
    W_n(x) = \frac{1}{n^{d/(2\alpha + d)}}\sum_{l=1}^\infty l^{-(\frac{\alpha}{d}-\frac{1}{2})}  w_l \psi_l(x),
    \quad x\in\cube,
    \quad \alpha>d,
    \quad w_l\iid \text{Laplace},
\end{equation}
with $(\psi_l, \ l\in\N)$ a wavelet basis of $L^2([0,1]^d)$ generating the scale of Besov spaces, cf.~Section \ref{Subsec:Notation}, and where the Laplace distribution has p.d.f.~proportional to $e^{-|t|/2}$, $t\in\R$. In \cite{ADH21}, the law of $W_n$ in \eqref{Eq:BaseLaplPrior} is called a `rescaled $(\alpha-d)$-regular Laplace prior' in view of the fact that its realizations belong almost surely to $B^\beta_p(\cube)$ for all $\beta<\alpha - d$ and all $p\in[1,\infty]$, cf.~\cite[Lemma 5.2]{ADH21}. In the terminology of \cite{LSS09}, \eqref{Eq:BaseLaplPrior} defines a (rescaled) `$B^\alpha_1$-prior'.

In the aforementioned frequentist analysis literature, the smoothness hyper-parameter $\alpha$ was shown to drive the speed of concentration of posterior distributions associated to rescaled Besov-Laplace priors, requiring a precise matching to the regularity of the ground truth to achieve minimax-optimal posterior contraction rates. Since assuming knowledge of the true smoothness is typically unrealistic, here we seek a fully data-driven procedure able to automatically `adapt' to it, achieving optimal performances for a wide range of regularities. To do so, we employ the hierarchical Bayesian approach, e.g.~\cite[Chapter 10]{GvdV17}, assigning a hyper-prior to $\alpha$ in \eqref{Eq:BaseLaplPrior}. Specifically, we model $\alpha\sim\Sigma_n$, where $\Sigma_n$ is an $n$-dependent absolutely continuous distribution supported on the interval $(d,\log n]$, with p.d.f.
\begin{equation}
\label{Eq:Hyperprior}
    \sigma_n(\alpha) = \frac{e^{-n^{d/(2\alpha+d)}}}{\zeta_n},
    \qquad \alpha\in (d,\log n],
\end{equation}
whose normalization constant satisfies $\zeta_n\simeq \log n$. This construction is motivated by previous findings in the literature showing that such hyper-prior distributions enjoy certain universal adaptation properties. See \cite{lember2007universal} for results in density estimation, and \cite{vWvZ16} for drift estimation for diffusion processes. An analogous choice also underpins the adaptive posterior contraction rates for Besov-Laplace priors in density estimation derived by \cite{giordano23besov}. See Remark \ref{Rem:nDependence} for further discussion.

For $W_n$ as in \eqref{Eq:BaseLaplPrior}, and $\alpha\sim \Sigma_n$ with hyper-prior p.d.f.~as in \eqref{Eq:Hyperprior}, we conclude the construction of the prior distribution $\Pi$ for probability response functions, whose range is equal to $[0,1]$, via a transformation through a smooth and strictly increasing link function $H:\R\to[0,1]$. For concreteness, we take the logistic (or `sigmoid') link $H(t) = e^t/(e^t+1)$, $t\in\R$, and let $\Pi_n$ be the law of the random function
\begin{equation}
\label{Eq:FinalPrior}
    p_{W_n} (x) := H[W_n(x)] = \frac{e^{W_n(x)}}{e^{W_n(x)} + 1}, \qquad x \in [0,1]^d.
\end{equation}
Other links (such as the probit $H = \Phi$, with $\Phi$ the standard normal cumulative distribution function) could be used as well under regularity conditions. We call $\Pi_n$ a hierarchical rescaled Besov-Laplace prior for classification surfaces, where the subscript $n$ highlights the dependence of the prior on the sample size, cf.~Remarks \ref{Rem:Resc} and \ref{Rem:nDependence}.

Given labeled binary classification data $D^{(n)}$ from \eqref{Eq:Model}, the posterior distribution $\Pi_n(\cdot|D^{(n)})$ of $p|D^{(n)}$ is then given by Bayes'~formula, 
\begin{equation}
\label{Eq:Post}
    \Pi_n(A|D^{(n)})
    =\frac{\int_A L^{(n)}(p)d\Pi_n(p)}
    {\int_\Pcal L^{(n)}(p)d\Pi_n(p)},
    \qquad A\subseteq\Pcal\ \text{measurable},
\end{equation}
cf.~\cite[p.7]{GvdV17}, with $L^{(n)}$ the likelihood from \eqref{Eq:Likelihood}. Since the latter depends on $\mu_X$ only through the multiplicative terms $\mu_X(X_i)$, $\Pi_n(\cdot|D^{(n)})$ is independent of $\mu_X$. Following the Bayesian paradigm, $\Pi_n(\cdot|D^{(n)})$ represents the updated belief about $p$ after observing the data and furnishes point estimates and uncertainty quantification. See Section \ref{Sec:Simulations} for a concrete illustration with synthetic data.

\begin{remark}[Rescaling]
\label{Rem:Resc}
Similar rescaling terms to the sequence $n^{-d/(2\alpha+d)}$ introduced in \eqref{Eq:BaseLaplPrior} underpin essentially all existing frequentist analyses of Besov-Laplace priors, e.g.~\cite{ADH21,giordano23besov,agapiou2024laplace}. By uniformly shrinking the prior draws, this enforces additional regularization and yields tight complexity bounds for a properly chosen sequence of `sieves', that is, subsets of the parameter space of high prior probability, that play a crucial role in the pursuit of the testing approach to posterior contraction rates, e.g.~\cite{GvdV17}. See the discussion after Theorem 1 in \cite{giordano23besov} for further insights. In a recent investigation by \cite{dolera2024strong},  it was shown via a different proof strategy that non-rescaled (in particular, $n$-independent) Besov-Laplace priors can attain optimal posterior contraction rates. However, these results only apply to the simpler white noise model, and it remains unclear whether they could be extended to the binary classification problem at hand.
\end{remark}

\begin{remark}[$n$-dependent priors]
\label{Rem:nDependence}

The above hierarchical rescaled Besov-Laplace prior $\Pi_n$ depends on the sample size through both the rescaling sequence in \eqref{Eq:BaseLaplPrior} and the hyper-prior choice in \eqref{Eq:Hyperprior}. This is a departure from a genuine single projective Bayesian model, motivated by our asymptotic analysis, where it leads to optimal frequentist adaptation properties, similarly to the previous findings of \cite{lember2007universal, vW19, giordano23besov}, among the others. Alternative hierarchical priors, based on $n$-independent gamma-type hyper-priors, have been studied in \cite{agapiou2024adaptive} for the more tractable white noise model. Obtaining adaptive posterior contraction rates with these arguably more natural constructions in the present setting is an interesting open problem for future research.

\end{remark}

%
%
%

\subsection{Adaptive posterior contraction rates}
\label{Subsec:LaplRates}

In this section, we characterize the asymptotic behavior of the posterior distribution as $n\to\infty$ under the frequentist assumption that the data $D^{(n)}\sim Q_{p_0}^{(n)}$ have been generated as in \eqref{Eq:Model} by some possibly spatially inhomogeneous true probability response function $p_0\in B_1^{\alpha_0}(\cube)$, for some $\alpha_0> d$. The following result quantifies, via the notion of `posterior contraction rates', cf.~\cite[Chapter 8]{GvdV17}, the speed at which $\Pi_n(\cdot|D^{(n)})$ concentrates around $p_0$ in $L^1$-distance.

\begin{theorem}
\label{Theo:LaplRates}
    Let $\Pi_n$ be a hierarchical rescaled Besov-Laplace prior for probability response functions constructed as in Section \ref{Subsec:Prior}. Let $D^{(n)} = \{(Y_i, X_i)\}^n_{i=1}\sim Q_{p_0}^{(n)}$ be a random sample of labeled binary classification data arising from  \eqref{Eq:Model} for some fixed $p = p_0 \in B^{\alpha_0}_1([0,1]^d)$ for some $\alpha_0>d$, satisfying $\inf_{x\in[0,1]^d} p_0(x) > 0$. Then, for $M>0$ large enough, as $n \to \infty$,
    \begin{equation*}
        E^{(n)}_{p_0} \left[\Pi_n \left( p : \| p - p_0 \|_1 > M n^{-\frac{\alpha_0}{2 \alpha_0 + d}} \Big|  D^{(n)} \right)\right] \to 0.
    \end{equation*}
\end{theorem}

Theorem \ref{Theo:LaplRates} entails that, with $Q_{p_0}^{(n)}$-probability tending to one, $\Pi_n(\cdot|D^{(n)})$ puts all of its probability mass in small neighborhoods of $p_0$ with $L^1$-radius shrinking at rate $n^{-\alpha_0/(2 \alpha_0 + d)}$. Consequently, draws from the posterior distribution provide increasingly better reconstructions of the ground truth.

The rate $n^{-\alpha_0/(2 \alpha_0 + d)}$ is known to be optimal, in the minimax sense, for the $L^1$-distance over the Besov space $B^{\alpha_0}_1(\cube)$, e.g.~\cite{DJ98}. In Theorem \ref{Theo:LaplRates}, this is achieved via a nonparametric Bayesian procedure that does not require knowledge of the true regularity, but rather adapts to $\alpha_0$ in the wide range $(d,\infty)$. We then conclude that hierarchical rescaled Besov-Laplace priors achieve adaptive posterior contraction rates in binary classification. This is in line with the existing adaptation results for hierarchical Besov-Laplace priors in density estimation, \cite{giordano23besov}, and in the white noise model, \cite{agapiou2024adaptive}.

For spatially homogeneous true probability response functions belonging to traditional H\"older spaces, \cite{vdVvZ09} proved adaptation for hierarchical Gaussian priors with random length-scale. Our result provides a parallel to this for ground truths in the $B^\alpha_1$-Besov scale and hierarchical rescaled Besov-Laplace priors. On the other hand,  Gaussian priors have been shown to be unable to optimally reconstruct spatially inhomogeneous functions, \cite{agapiou2024laplace}, and thus cannot be expected to achieve optimal posterior contraction rates, even non-adaptive ones, in the setting of Theorem \ref{Theo:LaplRates}, regardless of any specific tuning or randomization. We provide a numerical illustration of this phenomenon in Section \ref{Sec:Simulations} via synthetic data.

\begin{remark}[Boundedness away from zero]
In Theorem \ref{Theo:LaplRates}, the assumption that the ground truth be positive throughout the domain guarantees that $p_0$ is in the range of the composition with respect to the link function $H$. Since, by construction, the prior $\Pi_n$ is supported over such functional compositions, this restriction appears to be necessary for posterior consistency in the considered setting. Similar assumptions are also imposed for the frequentist analysis of Gaussian priors for binary classification developed by \cite{vdVvZ08}. We refer the reader to \cite[Chapter 9.5.6]{GvdV17} for results for priors based on the Dirichlet process in the case where the true probability response function is not necessarily bounded away from zero. Investigating such scenario in the presence of spatial inhomogeneity is an interesting open question.
\end{remark}

%
%
%

\subsection{Posterior sampling}
\label{Subsec:Algorithm}

For the observation model \eqref{Eq:Model}, the posterior distribution resulting from the considered prior $\Pi_n$ is not available in closed form. We then employ an MCMC method to draw approximate samples from $\Pi_n(\cdot|D^{(n)})$ and concretely implement posterior inference.

Specifically, within a Gibbs-type scheme to handle the hierarchical construction, we resort to the `whitened pre-conditioned Crank-Nicolson' (wpCN) algorithm of \cite{CDPS18}, which is a Metropolis-Hastings-type technique applicable to prior distributions that can be expressed (in our case, conditionally) as transformation of a Gaussian white noise. For the prior $\Pi_n$ from Section \ref{Subsec:LaplRates}, we observe that, for fixed $\alpha$, the random function $W_n$ in \eqref{Eq:BaseLaplPrior} is equal in distribution to
\begin{equation}
\label{Eq:WhiteXi}   
    T^{(n)}_{\alpha}(\xi)(x) := \sum_{l=1}^\infty T^{(n)}_{\alpha,l}(w_l) \psi_l(x), 
    \qquad x\in\cube
    \qquad w_l\iid N(0,1),
\end{equation}
where $\xi$ is a Gaussian white noise process indexed by $[0,1]^d$ given by
\begin{equation}
\label{Eq:Xi} 
    \xi(x) := \sum_{l=1}^\infty w_l \psi_l(x),
    \qquad x\in\cube,
    \qquad w_l\iid N(0,1),
\end{equation}
and where the `whitening transformation' $T^{(n)}_\alpha$ is defined by
$$
    T^{(n)}_{\alpha,l}(w) 
    := n^{-\frac{d}{2\alpha+d}}l^{-(\frac{\alpha}{d}-\frac{1}{2})} 
    \sgn(w) \left[ -\log(2 - 2\Phi(|w|) \right],
    \qquad w\in\R,
    \qquad l\in\N.
$$

Starting from some initialization for $\alpha$ (e.g.~a fixed `cold start' $\alpha =d+1$), and given an initial white noise sample $\xi_0$ (obtained from \eqref{Eq:Xi} by drawing independent standard normal random coefficients), with $\omega_0 := T^{(n)}_{\alpha_0}(\xi_0)$ the corresponding initialization for $w=H^{-1}\circ p$, each step of the wpCN-within-Gibbs algorithm alternates samples from:
\begin{enumerate}
    \item The full conditional distribution of the smoothness parameter $\alpha$, via a standard random walk Metropolis-Hastings algorithm, namely:
    \begin{itemize}
        \item propose $\alpha_* := \max\{\alpha_{s-1} + \delta_1 Z,d\}$, where $\delta_1 > 0$ is a fixed step-size and $Z$ is an independent standard Gaussian random variable.
        \item Set
        $$
            \alpha_s:=
            \begin{cases}
                \alpha_*, & \textnormal{with probability}\ \min\left\{1, \frac{L^{(n)}(H\circ T_{\alpha_*}^{(n)}(\xi_{s-1}))}
                {L^{(n)}(H\circ T_{\alpha_{s-1}}^{(n)}(\xi_{s-1}))} \times \frac{\sigma_n( \alpha_*)}{\sigma_n(\alpha_{s-1})}\right\},\\
                \alpha_{s-1}, & \text{otherwise},
            \end{cases}
        $$
    with $L^{(n)}$ the likelihood from \eqref{Eq:Likelihood} and $\sigma_n$ the hyper-prior p.d.f.~from \eqref{Eq:Hyperprior}.
    \end{itemize}

    \item The full conditional distribution of the infinite-dimensional parameter $\omega$ via the wpCN algorithm, namely:
    \begin{itemize}
        \item Construct the whitened proposal $\xi^* :=\sqrt{1-2\delta_2}$ $\xi_{s-1} + \sqrt{2\delta_2} \chi$, where $\delta_2\in(0,1/2)$ is a fixed step-size and $\chi$ is an independent Gaussian white noise.
        \item Set
        $$
            \xi_{s}:=
            \begin{cases}
                \xi^*, & \textnormal{with probability}\ \min\left\{1, \frac{L^{(n)}(H\circ T_{\alpha_s}^{(n)}(\xi^*))}
                {L^{(n)}(H\circ \omega_{s-1})}\right\},\\
                \xi_{s-1}, & \text{otherwise}.
            \end{cases}
        $$
        \item Set $\omega_{s} = T^{(n)}_{\alpha_s}(\xi_{s})$.
    \end{itemize}
\end{enumerate}

In practice, we implement the first step above by truncating the series in \eqref{Eq:WhiteXi} and \eqref{Eq:Xi} at some pre-specified level $L\in\N$, sufficiently high as to guarantee that the resulting numerical approximation error is negligible compared to the statistical one (e.g.~taking $L$ proportional to $n$). The second operation necessitates the evaluation of the proposal likelihood, which is straightforward for the observation model \eqref{Eq:Model}, cf.~\eqref{Eq:Likelihood}. We note that as $L^{(n)}$ depends on the covariate p.d.f.~$\mu_X$ only through the multiplicative terms $\mu_X(X_i)$, computing the acceptance probabilities does not require knowledge of $\mu_X$.

The resulting Markov chain $(\alpha_s,\omega_s)_{s=0}^\infty$ has limiting distribution equal to the joint posterior distribution of $(\alpha,w)$, \cite{CDPS18}. Moreover, the underlying pCN-type structure is known to give rise to dimension-robust acceptance probabilities, \cite{CRSW13}. This implies desirable mixing properties and efficient convergence towards equilibrium, even under high discretization dimensions (i.e., truncation levels).

\begin{remark}[Computational complexity]
The computational complexity of each step within the above MCMC scheme is driven by the proposal of the whitened parameter $\xi^*$ and the evaluation of two likelihood ratios for the acceptance probabilities. The cost of these operations indirectly depends on the covariate dimension $d$, through the choice of the truncation level $L$ for the series \eqref{Eq:WhiteXi} and \eqref{Eq:Xi}. For standard multi-resolution wavelets, typically $L = 2^{kd}$ for some $k < K$, where $K$ is the higher scaling resolution. By pre-computing the wavelet values $\psi_l(X_i)$ at each observed covariate, evaluating likelihoods of the form \eqref{Eq:Likelihood} then reduces to multiplications of $L$-dimensional vectors (of wavelet coefficients) by an $L\times n$ design matrix.

For high-dimensional applications, where this could still represent an important computational bottleneck, geometric sparsity properties of specific wavelet bases, like compactly supported Daubechies wavelets or Haar functions, could allow to mitigate such exponential dependencies. By detecting `active nodes', namely ones where $\psi_l(X_i) \neq 0$, the proposal can be reduced to only the corresponding `active' coefficients, and sparse matrix algebra can be deployed to handle the resulting sparse pre-computed wavelet design matrix. We did not explicitly pursue such computational optimisations in the present work.
\end{remark}

%
%
%

\section{Simulation studies}\label{Sec:Simulations}

We assess the empirical performance of the considered hierarchical rescaled Besov-Laplace prior for binary classification via numerical experiments. For one- and bi-dimensional domains $[0,1]^d$, $d=1,2$, we fix spatially homogeneous and inhomogeneous true probability response functions $p_0$, generate independent and identically distributed (i.i.d.) inputs $X_i\iid\text{Uniform}([0,1]^d)$, conditionally on which we  sample the labels $Y_1,\dots,Y_n$ according to model \eqref{Eq:Model}. We then perform posterior inference via the wpCN-within-Gibbs MCMC algorithm for posterior sampling outlined in Section \ref{Subsec:Algorithm}.

For comparison, we also consider a hierarchical Gaussian prior defined via a centered stationary Gaussian process $G:=(G_x, \ x\in[0,1]^d)$ with square-exponential covariance kernel, whose length-scale is assigned an inverse-Gamma hyper-prior,
\begin{align*}
    E\left[ G_xG_y|A \right] = e^{-A|x - y|^2},
    \qquad x,y\in\cube,
    \qquad A^d\sim \Gamma(1,1).
\end{align*}
Combined with suitable link functions $H$ (including e.g.~the logistic one), this was shown by \cite{vdVvZ09} to achieve adaptive posterior contraction rates towards true probability response functions with traditional H\"older regularity. However, it is generally expected to be unable to optimally recover spatially inhomogeneous ground truths, in view of the known sub-optimality of Gaussian priors in this case, cf.~the discussion after Theorem \ref{Theo:LaplRates}.

Posterior inference with the above hierarchical Gaussian prior is implemented via a Metropolis-within-Gibbs MCMC sampling algorithm similar to the one from Section \ref{Subsec:Algorithm}, where the wpCN routine for Besov-Laplace priors used therein is replaced by the standard pCN method, \cite{CRSW13}. All the numerical experiments were carried out in \texttt{R} on an Intel(R) Core(TM) i7-10875H 2.30GHz processor with 32 GB of RAM. Each MCMC run was iterated for 25,000 steps, discarding the first 10,000 as burn-in. Maximum per-experiment computation times were of around 20 minutes.

%
%
%

\subsection{Univariate experiments}\label{Subsec:1DExp}

We start considering a univariate spatially homogeneous scenario, setting
\begin{equation}
\label{Eq:1DTruth}
    p_0(x) = \frac{1}{1 + e^{- (9x - 5)}},
    \qquad x\in[0,1],
\end{equation}
namely a scaled and shifted sigmoid function, restricted to the unit interval $[0,1]$. See Figure \ref{Fig:1DEstim}. With this setup, we sampled labeled binary classification data $D^{(n)}$ from model \eqref{Eq:Model}, with $\mu_X = \text{Uniform}([0,1])$ and increasing sample sizes $n = 50,200,1000,5000$. The observations are represented by the rugs in Figure \ref{Fig:1d_block}.

We then performed posterior inference based on the hierarchical rescaled Besov-Laplace priors from Section \ref{Subsec:Prior} and the hierarchical Gaussian prior described above (for which the logistic link function was also chosen). Figure \ref{Fig:1DEstim} shows the obtained (MCMC approximations to the) `posterior means' $\bar p_n := H\circ E^{\Pi_n}[w|D^{(n)}]$ for $n = 200,1000,5000$. Uncertainty is quantified and visualized by the associated point-wise $95\%$-credible intervals. As expected from Theorem \ref{Theo:LaplRates} and the existing theory for hierarchical Gaussian priors, \cite{vdVvZ09}, both procedures achieve convergence towards the spatially homogeneous ground truth, with practically indistinguishable reconstructions at the largest sample size. This visual comparison is corroborated by Table \ref{Tab:1DEstim}, where $L^1$-estimation errors, averaged over 50 replications of each experiment, are reported, jointly with the corresponding standard deviations.

For the Besov-Laplace prior, we used Daubechies-8 maximally symmetric (i.e.~`Symmelet-8') functions, with symmetric boundary reflection, implemented in the $\texttt{R}$ package $\texttt{wavethresh}$, truncating the series after the first $1,024$ terms. All MCMC runs were initialized at cold, uninformative starts. The step sizes for the wpCN and pCN routines were chosen within the range $[0.001,0.05]$, depending on the sample size, to achieve a stabilization of the acceptance probabilities at around $30\%$ after burn-in.  

\begin{figure}[H]
	\centering
	\includegraphics[width=\linewidth]{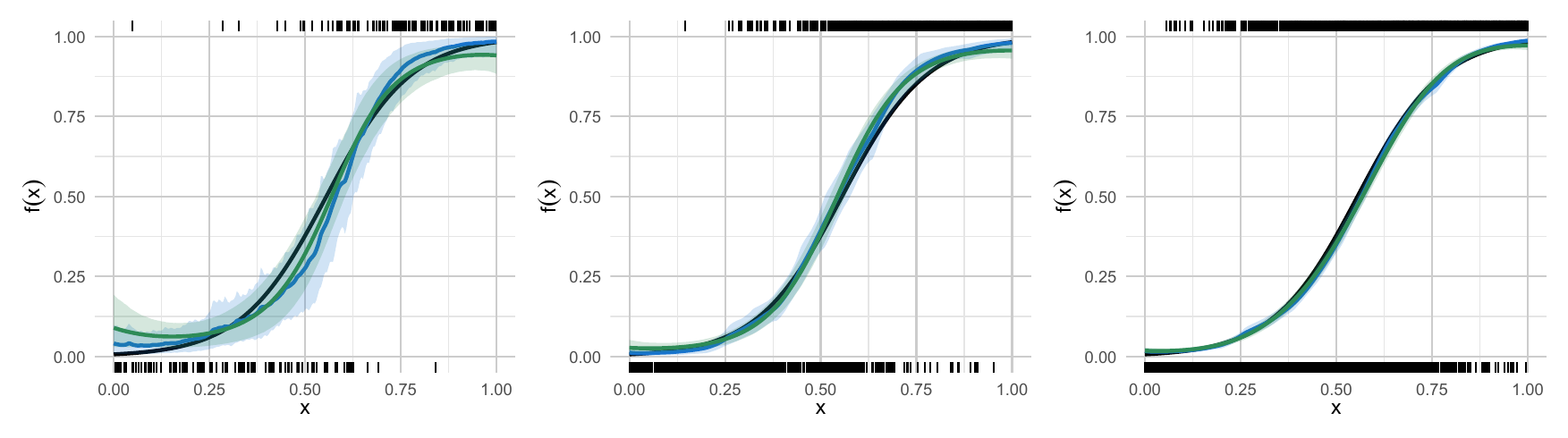}
	\caption{Left to right: Posterior means for Gaussian (solid green) and Besov-Laplace (solid blue) priors, pointwise $95\%$-credible intervals (shaded regions of the corresponding colors) for $n = 200, 1000, 5000$, respectively. The ground truth $p_0$ from \eqref{Eq:1DTruth} is shown in solid black. Rugs at the bottom represent the covariate values labeled $0$, while rugs at the top represent covariates labeled $1$.}
    \label{Fig:1DEstim}
\end{figure}

\begin{table}[!ht]
    \centering
    \begin{tabular}{rr|rrrrr}
    & $n$ & 50 & 200 & 1000 & 5000  \\  
    \hline
    \multirowcell{2}[0pt][l]{Gaussian} & $\| \bar p_n - p_0 \|_1$ 
    & 0.16 (0.02) & 0.05 (0.01) & 0.02 (0.005) & 0.01 (0.003) \\
    & $\| \bar p_n - p_0  \|_1/
    \| p_0\|_1$ 
    & 0.27 (0.03) & 0.09 (0.03) & 0.04 (0.008) & 0.02 (0.005)  \\ 
    \hline
    \multirowcell{2}[0pt][l]{Laplace} & $\| \bar p_n - p_0 \|_1$ 
    & 0.21 (0.05) & 0.05 (0.02) & 0.02 (0.007) & 0.01 (0.003) \\
    & $\| \bar p_n - p_0  \|_1/
    \| p_0\|_1$ 
    & 0.37 (0.09) & 0.09 (0.03) & 0.04 (0.011) & 0.02 (0.005)  \\ 
    \hline
    \end{tabular}
    \caption{Average $L^1$-estimation errors for the posterior mean (and their standard deviations) over 50 repeated experiments with the ground truth $p_0$ from \eqref{Eq:1DTruth}.}
    \label{Tab:1DEstim}
\end{table}

We next consider the step-like spatially inhomogeneous ground truth
\begin{equation}
    p_0(x) = \begin{cases}
    0.9 \qquad x\in[0,0.4) \\
    0.2 \qquad x\in[0.4,0.7) \\
    0.5 \qquad x\in[0.7,1],
    \end{cases}
    \qquad x\in[0,1]
    \label{Eq:1D_block}
\end{equation}
cf.~Figure \ref{Fig:1d_block} below, for which the obtained results are summarized in Figure \ref{Fig:1d_block} below. Unlike the preceding spatially homogeneous scenario, here a marked difference between the performance of the two procedures emerges. In particular, the hierarchical Gaussian prior appears to be unable to correctly detect the sharp jumps at inputs $x = 0.4,0.7$, significantly over-smoothing the edges of the blocks even as the sample size increases. This is in line with the known sub-optimality of Gaussian priors, \cite{agapiou2024laplace}, and generally of linear procedures, \cite{DJ98}, in the presence of spatial inhomogeneity. On the contrary, the hierarchical rescaled Besov Laplace priors achieves progressively more faithful reconstructions. Table \ref{Tab:1d_block} reports  the obtained average $L^1$-estimation errors for the posterior mean, where the latter procedure is shown to outperform the former across all values of $n$, with a more significant improvement in the recovery at the largest sample size. 

\begin{figure}[H]
\centering
\includegraphics[width=\linewidth]{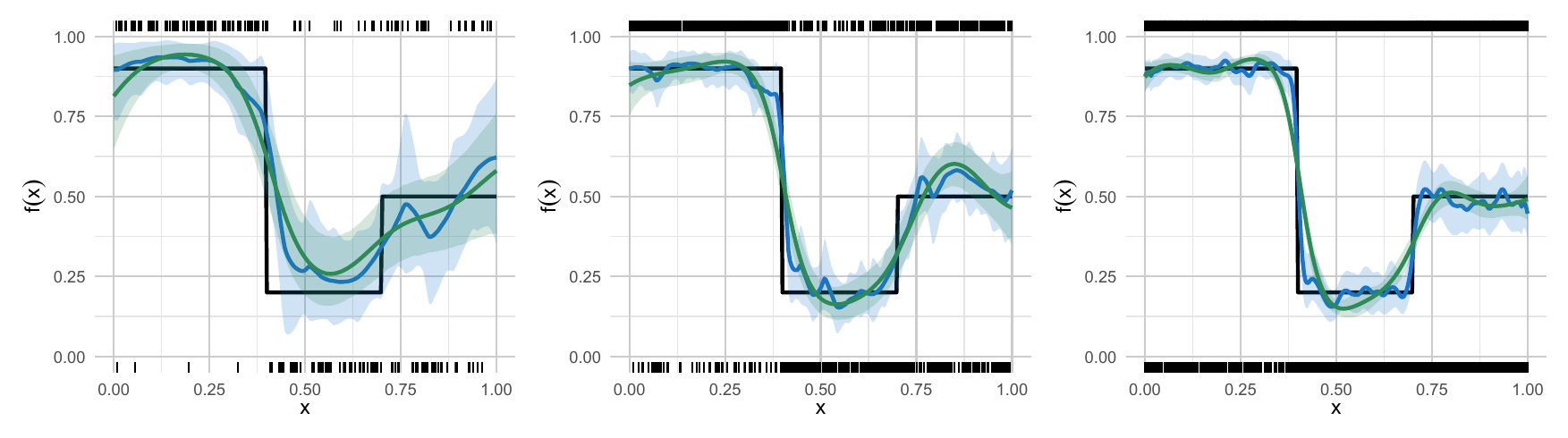}
\caption{Left to right: Posterior means for Gaussian (solid green) and Laplace (solid blue) priors, pointwise $95\%$-credible intervals (shaded regions) for $n = 200, 1000, 5000$, respectively. The ground truth $p_0$ from \eqref{Eq:1D_block} is shown in solid black.}
\label{Fig:1d_block}
\end{figure}

\begin{table}[!ht]
    \centering
    \begin{tabular}{rr|rrrrr}
    & $n$ & 50 & 200 & 1000 & 5000  \\  
    \hline
    \multirowcell{2}[0pt][l]{Gaussian} & $\| \bar p_n - p_0 \|_{L^1}$ 
    &0.26 (0.03) & 0.17 (0.03) & 0.09 (0.01) & 0.08 (0.006) \\
    & $\| \bar p_n - p_0  \|_{L^1}/
    \| p_0\|_{L^1}$ 
    & 0.36 (0.05) & 0.19 (0.04) & 0.14 (0.01) & 0.12 (0.009)  \\ 
    \hline
    \multirowcell{2}[0pt][l]{Laplace} & $\| \bar p_n - p_0 \|_{L^1}$ 
    & 0.25 (0.03) & 0.15 (0.03) & 0.09 (0.01) & 0.06 (0.003) \\
    & $\| \bar p_n - p_0  \|_{L^1}/
    \| p_0\|_{L^1}$ 
    & 0.35 (0.04) & 0.17 (0.04) & 0.14 (0.01) & 0.09 (0.005)  \\ 
    \hline
    \end{tabular}
    \caption{Average $L^1$-estimation errors for the posterior mean (and their standard deviations) over 50 repeated experiments with the ground truth $p_0$ from \eqref{Eq:1D_block}.}
    \label{Tab:1d_block}
\end{table}

%
%
%

\subsection{Bivariate experiments}\label{Subsec:2DExp}

Over the unit square $[0,1]^2$, we also consider two scenarios, respectively with:
\begin{enumerate}
\item A spatially homogeneous true probability response function given by
\begin{align}
\label{Eq:2D_skn}
    f(x_1,x_2) = \frac{1}{4} \ f_{SN}\left(x_1,x_2; \ (0.4, 0.6), \ 0.05 I_2, \ (3,-2) \right),
    \qquad (x_1,x_2)\in[0,1]^2,
\end{align}
where $f_{SN}$ denotes the (bivariate) skew-normal p.d.f.~and $I_2$ is the identity matrix of $\R^{2,2}$. See the last panel of Figure \ref{Fig:2D_skn};
\item A spatially inhomogeneous ground truth with a square block component
\begin{equation}
\label{Eq:2D_bs}
    p_0(x_1,x_2) = 0.4 \prod_{h = 1}^2 (1 + \sgn(x_h - b_h)) (1 - \sgn(x_h - c_h)),
    \qquad (x_1,x_2)\in[0,1]^2,
\end{equation}
where the extremes are $b = (0.1, 0.1)$ and $c = (0.5, 0.5)$, cf.~Figure \ref{Fig:2D_bs} (last panel).
\end{enumerate}

For both, the obtained posterior mean estimates $\bar p_n$ based on the hierarchical rescaled Besov-Laplace prior and the hierarchical Gaussian prior are shown in Figures \ref{Fig:2D_skn} and \ref{Fig:2D_bs}, respectively, for increasing sample sizes $n = 200, 1000, 5000$. The associated $L^1$-estimation errors are reported in Tables \ref{Tab:2D_skn} and \ref{Tab:2D_bs}, respectively.

The numerical results broadly reinforce the findings from the univariate experiments from Section \ref{Subsec:1DExp}. In the homogeneous case, both procedures deliver excellent recoveries, achieving very similar estimation errors, that steadily decrease as $n$ grows. On the other hand, the hierarchical Gaussian prior appears to be unable to correctly reconstruct the edges of the block component of the true probability response function \eqref{Eq:2D_bs}, which are visibly oversmoothed. These are instead precisely reconstructed by the hierarchical rescaled Besov-Laplace prior. For the latter, the obtained estimation errors are lower across all values of $n$, and particularly for the largest sample size.

\begin{figure}[H]
	\centering
    \subfloat{\includegraphics[width=\linewidth]{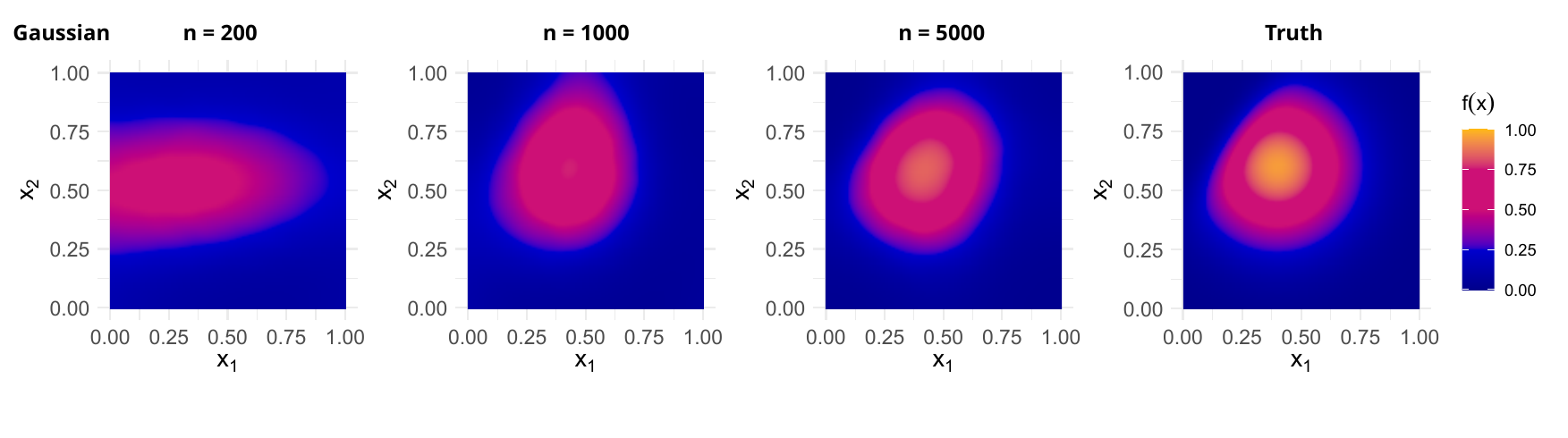}}\\
    \vspace{-2em}
    \subfloat{\includegraphics[width=\linewidth]{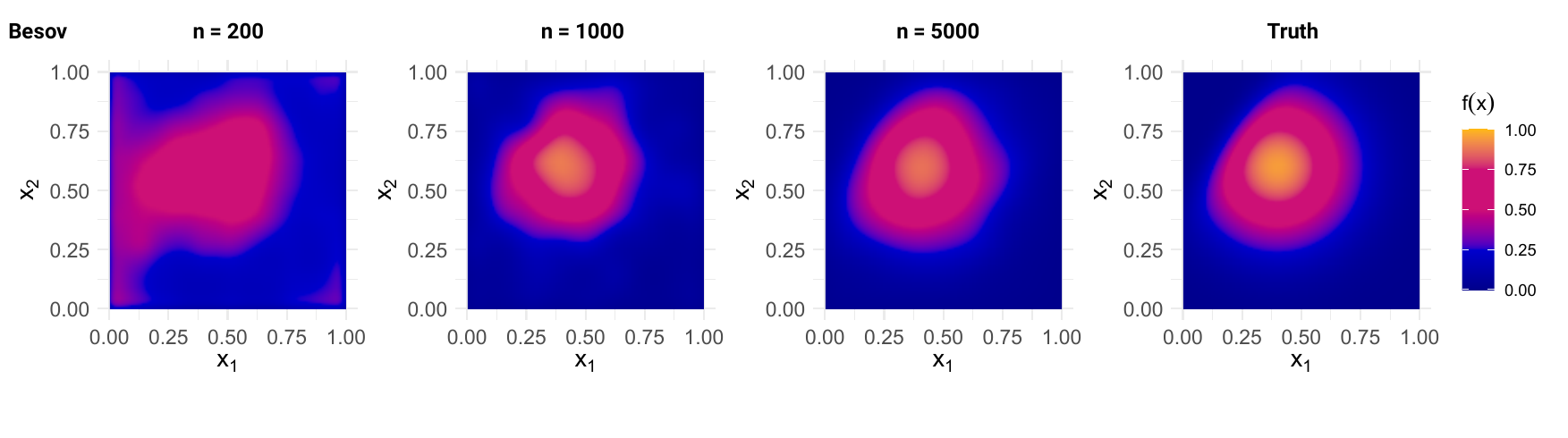}}
	\caption{Top to bottom, left to right: posterior means for Gaussian (top) and Besov-Laplace (bottom) priors for increasing sample sizes $n = 200,1000,5000$, in case of the spatially homogeneous ground truth \eqref{Eq:2D_skn}, shown in the rightmost panel of both rows.}
    \label{Fig:2D_skn}
\end{figure}

\begin{table}[!ht]
    \centering
    \begin{tabular}{rr|rrrrr}
    & $n$ & 50 & 200 & 1000 & 5000  \\  
    \hline
    \multirowcell{2}[0pt][l]{Gaussian} & $\| \bar p_n - p_0 \|_{L^1}$ 
    & 0.26 (0.01) & 0.18 (0.03) & 0.07 (0.006) & 0.04 (0.003) \\
    & $\| \bar p_n - p_0  \|_{L^1}/
    \| p_0\|_{L^1}$ 
    & 0.75 (0.02) & 0.50 (0.09) & 0.19 (0.02) & 0.12 (0.01)   \\ 
    \hline
    \multirowcell{2}[0pt][l]{Laplace} & $\| \bar p_n - p_0 \|_{L^1}$ 
    & 0.28 (0.03) & 0.17 (0.01) & 0.07 (0.007) & 0.05 (0.003)  \\
    & $\| \bar p_n - p_0  \|_{L^1}/
    \| p_0\|_{L^1}$ 
    & 0.80 (0.06) & 0.48 (0.02) & 0.21 (0.02) & 0.13 (0.01)  \\ 
    \hline
    \end{tabular}
    \caption{Average $L^1$-estimation errors for the posterior mean (and their standard deviations) over 50 repeated experiments with the ground truth $p_0$ from \eqref{Eq:2D_skn}.}
    \label{Tab:2D_skn}
\end{table}

\begin{figure}[H]
\centering
\subfloat{\includegraphics[width=\linewidth]{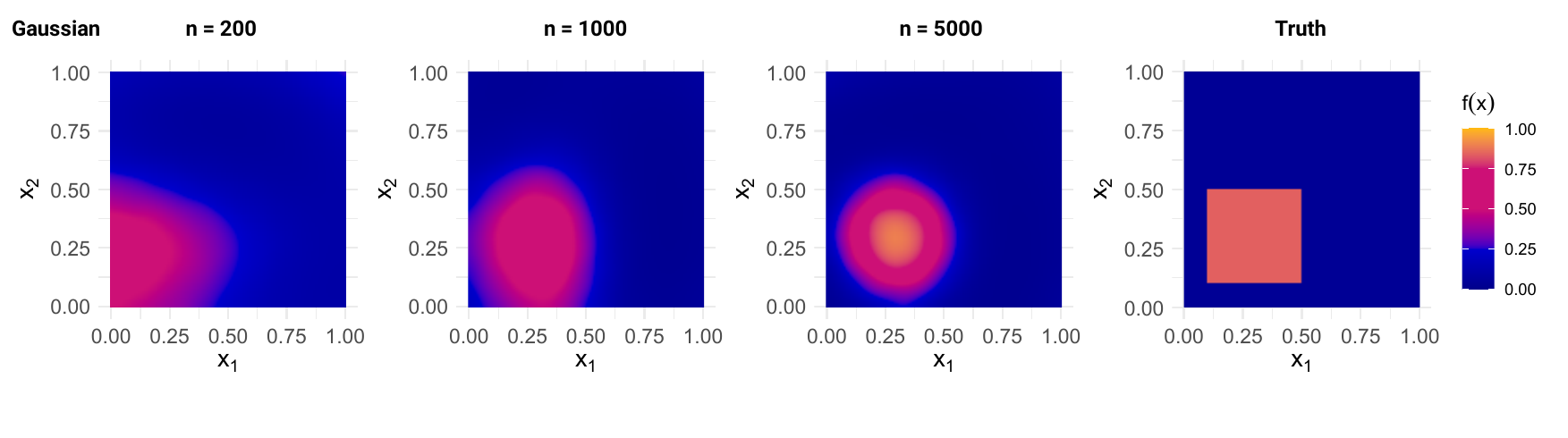}}\\
    \vspace{-2.5em}
    \subfloat{\includegraphics[width=\linewidth]{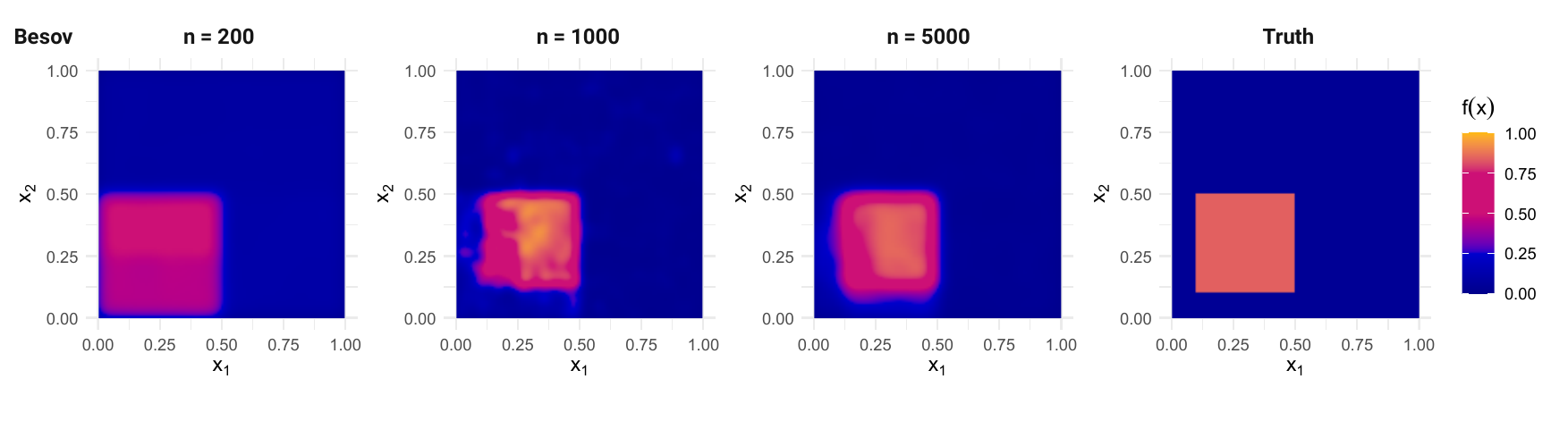}}
\caption{Top to bottom, left to right: posterior means for Gaussian (top) and Besov-Laplace (bottom) priors for increasing sample sizes $n = 200,1000,5000$, in case of the spatially inhomogeneous ground truth \eqref{Eq:2D_bs}, shown in the rightmost panel of both rows.}
\label{Fig:2D_bs}
\end{figure}

\begin{table}[!ht]
    \centering
    \begin{tabular}{rr|rrrrr}
    & $n$ & 50 & 200 & 1000 & 5000  \\  
    \hline
    \multirowcell{2}[0pt][l]{Gaussian} & $\| \bar p_n - p_0 \|_{L^1}$ 
    & 0.28 (0.05) & 0.24 (0.02) & 0.17 (0.01) & 0.14 (0.003) \\
    & $\| \bar p_n - p_0  \|_{L^1}/
    \| p_0\|_{L^1}$ 
    & 0.84 (0.12) & 0.72 (0.05) & 0.51 (0.03) & 0.42 (0.009)  \\ 
    \hline
    \multirowcell{2}[0pt][l]{Laplace} & $\| \bar p_n - p_0 \|_{L^1}$ 
    & 0.28 (0.05) & 0.21 (0.01) & 0.14 (0.02) & 0.08 (0.003)  \\
    & $\| \bar p_n - p_0  \|_{L^1}/
    \| p_0\|_{L^1}$ 
    & 0.84 (0.11) & 0.62 (0.02) & 0.50 (0.05) & 0.25 (0.008) \\ 
    \hline
    \end{tabular}
    \caption{Average $L^1$-estimation errors for the posterior mean (and their standard deviations) over 50 repeated experiments with the ground truth $p_0$ from \eqref{Eq:2D_bs}.}
    \label{Tab:2D_bs}
\end{table}

%
%
%
%
%

\section{Discussion}\label{Sec:Discussion}

We have studied a nonparametric Bayesian approach to the recovery of spatially inhomogeneous binary classification surfaces based on hierarchical rescaled Besov-Laplace priors. Our main result (Theorem \ref{Theo:LaplRates}) shows that, for ground truths $p_0$ in the $B^\alpha_{1}$-scale, the posterior contracts around $p_0$ at the optimal rate in $L^1$-distance, automatically adapting to the smoothness. For implementation, we have outlined an MCMC posterior sampling algorithm in Section \ref{Subsec:Algorithm}, which we have applied in the simulation studies of Section \ref{Sec:Simulations}, demonstrating the practical feasibility and effective performance of the considered procedure. We conclude by reviewing several related research questions.

Firstly, as noted in Section \ref{Subsec:Prior}, the employed ($n$-dependent) prior is constructed from base rescaled Besov-Laplace random elements, cf.~\eqref{Eq:BaseLaplPrior}, and involves a carefully tuned hyper-prior for the regularity, cf.~\eqref{Eq:Hyperprior}. Recently, \cite{agapiou2024adaptive} achieved adaptive posterior contraction rates towards spatially inhomogeneous ground truths in the white noise model by assigning somewhat more natural independent hyper-priors both to the smoothness and the scaling terms. It remains unclear at the present stage whether their results could be extended to the binary classification setting.

The considered compact covariate space is a standard framework for theoretical analysis. For unbounded covariates, say on $\R^d$, the construction of the prior could in principle be extended using wavelet bases of $L^2(\R^d)$. Under additional assumptions on the tail of the true probability response function, much of the key probabilistic properties derived in \cite{ADH21} would then be retained, paving the way to obtaining adaptive posterior contraction rates over Besov spaces of functions `decaying at infinity' similarly to those considered in \cite{Nickl2007BracketingME}.
Further, scenarios with both discrete and continuous predictors are often of interest. The present developments could be adapted to this case via product priors that separate the discrete and continuous components, modeling the probability response functions for the latter via Besov-Laplace priors. The hyper-prior on the regularity parameter $\alpha$ can then be refined to a hierarchical structure that allows for different smoothness levels across discrete categories, helping to achieve borrowing of strength across groups.

Lastly, let us also mention the important theoretical issue of providing a frequentist validation for the uncertainty quantification delivered by the considered Bayesian methodology, since it is known that, in infinite-dimensional statistical models, credible sets can have asymptotically zero coverage even if they arise from consistent posterior distributions; see~\cite{C93}. An established approach to derive such guarantees is via so-called `nonparametric Bernstein von-Mises theorems', \cite{CN13}. A result of this kind for Besov-Laplace prior has been recently established, in the context of drift estimation for diffusion processes, by \cite{giordano2025semiparametric}. Pursuing these results for the considered hierarchical priors for probability response functions is an interesting open question for future research.

\section*{Data availability statement}
The \texttt{R} code to replicate the simulation study and real data analyses is available at: \url{https://github.com/PatricDolmeta/Besov-Laplace-Classification}.

\section*{Acknowledgments}

The authors are grateful to the Associate Editor and two anonymous Referees for helpful comments that lead to an improvement of the manuscript.
The authors thank the ``de Castro'' Statistics Initiative and Collegio Carlo Alberto for supporting this research. M.G.~also acknowledges the partial financial support by MUR, PRIN project 2022CLTYP4. 



%
%
%



%
%
%
%
%
\bibliography{sn-bibliography.bib}


\clearpage 
\begin{center}
	\LARGE \textbf{Supplementary Material}
\end{center}
In this supplement, we present the proofs of all our results.

\appendix

\counterwithin{equation}{section}
\counterwithin{figure}{section}
\counterwithin{table}{section}

\section{Proofs}
\label{Sec:Proof}

%
%
%

\subsection{Proof of Theorem \ref{Theo:LaplRates}}

We apply the general program for posterior contraction rates in total variation distance set forth in \cite{GGvdV00}; see also \cite[Theorem 8.9]{GvdV17}. Recall from Section \ref{Sec:Model} that for a data pair $(X,Y)\sim Q^{(1)}_p$, $p\in\Pcal$, arising as in model \eqref{Eq:Model}, the joint probability density function is given by
$$
    q_p(y,x) = p(x)^y(1-p(x))^{1-y}\mu_X(x),
    \qquad y\in\{0,1\}, \qquad x\in\cube.
$$
For two probability response functions $p_1,p_2\in\Pcal$, the total variation distance between the laws $Q^{(1)}_{p_1}$, $Q^{(1)}_{p_2}$ is then given by
\begin{align*}
    TV(Q^{(1)}_{p_1},Q^{(1)}_{p_2}) 
    &=\frac{1}{2}\|q_{p_1} - q_{p_2}\|_{L^1(\{0,1\}\times\cube)}\\
    &=\frac{1}{2}\Bigg[\int_{\cube} |p_1(x) - p_2(x)|\mu_X(x)dx \\
    &\quad\ +\int_{\cube} |1-p_1(x) -1 + p_2(x)|\mu_X(x)dx\Bigg]
    = \|p_1 - p_2\|_{L^1(\cube,\mu_X)}
\end{align*}
cf.~eq.~(B.1) in \cite{GvdV17}. Since $\mu_X$ is bounded and bounded away from zero by assumption, we conclude that the latter is equivalent to the standard $L^1$-distance $\|p_1 - p_2\|_1$. Further, let 
$$
    K(p_1,p_2):=E^{(1)}_{p_1}\left[\log \frac{p_1(Y,X)}{p_2(Y,X)}\right];
    \qquad
    V(p_1,p_2):=E^{(1)}_{p_1}\left| \log \frac{p_1(Y,X)}{p_2(Y,X)} -  K(p_1,p_2)\right|^2,
$$
be the Kullback-Leibler divergence and variation, respectively. Theorem 8.9 in \cite{GvdV17} then yields that, if for some positive sequence $\varepsilon_n\to0$ such that $n\varepsilon_n^2\to\infty$, the hierarchical rescaled Besov-Laplace prior $\Pi_n$ for classification surfaces from Section \ref{Subsec:Prior} is shown to satisfy
\begin{equation}
\label{Eq:SmallBall}
	\Pi_n\left(p  :  K(p_0,p)\le \varepsilon_n^2,
    V(p_0,p)\le \varepsilon_n^2\right)
	\ge  e^{-Cn\varepsilon_n^2},
\end{equation}
for some constant $C>0$, and all $n\in\N$ large enough, as well as
\begin{equation}
\label{Eq:Sieves}
	\Pi_n(\Pcal_n^c)\le e^{-(C+4)n\varepsilon_n^2},
\end{equation}
for measurable sets $\Pcal_n\subseteq\Pcal$ satisfying
\begin{equation}
\label{Eq:MetricEntropy}
	\log N(\varepsilon_n; \Pcal_n, \|\cdot\|_1)
	\lesssim n\varepsilon_n^2,
\end{equation}
then $\Pi_n(\cdot|D^{(n)})$ contracts towards $p_0$ at rate $\varepsilon_n$ in total variation (i.e., standard $L^1$-) distance.

We verify conditions \eqref{Eq:SmallBall} - \eqref{Eq:MetricEntropy} with $\varepsilon_n = Mn^{-\alpha_0/(2\alpha_0+d)}$ with $M>0$ a large enough constant. Let $\Pi_{W_n}$ be the law of the hierarchical rescaled Besov-Laplace random element $W_n$ from \eqref{Eq:BaseLaplPrior} with smoothness hyper-prior $\alpha\sim\Sigma_n$ as in \eqref{Eq:Hyperprior}. Note that, since the logistic link $H$ is strictly increasing and smooth, it possesses a strictly increasing and smooth inverse, $H^{-1}:[0,1]\to\R$. Thus, in view of the positivity assumption $\inf_{x\in\cube}p_0(x)>0$, we have $p_0 =H\circ w_0$ for $w_0:= H^{-1}\circ p_0$. Further, recalling that $p_0 \in B^{\alpha_0}_{1}$, we can conclude by Theorem 10 in \cite{BS10}, that $w_0\in B^{\alpha_0}_{1}$ as well. Then, by Lemma 2.8 in \cite{GvdV17}, for all measurable and bounded $w:\cube\to\R$,
$$
    \max\left\{K(p_0,H\circ w),
    V(p_0,H\circ w)\right\}\lesssim \|w_0 - w\|_{L^2(\cube,\mu_X)}^2
    \simeq \|w_0 - w\|_2^2.
$$
The prior probability in \eqref{Eq:SmallBall} is then bounded below by, for some $c_1>0$, by
$$
	\Pi_{W_n}\left(w:\|w - w_0\|_2 \le c_1Mn^{-\frac{{\alpha_0}}{2{\alpha_0}+d}}
	\right)
$$
which, upon choosing $M>0$ large enough, is greater than $e^{-c_2 n^{d/(2\alpha_0+d)}} = e^{-c_2 n\varepsilon_n^2}$ for some $c_2>0$ by Lemma \ref{Lem:SmallBall} below.

Turning to the construction of the sieves $\Pcal_n$ from \eqref{Eq:Sieves} and \eqref{Eq:MetricEntropy}, by Lemma \ref{Lem:Sieves} below, there exists sufficiently large constants $b_1,b_2>0$ such the sets
$$
	\Wcal_n :=\left\{w = w^{(1)} + w^{(2)} 
    : \|w^{(1)}\|_1\le b_1 n^{-\frac{\alpha_*}{2\alpha_*+d}},  \
	\|w^{(2)}\|_{B^{\alpha_*+d}_{1}}\le b_1 n^\frac{d}{2\alpha_*+d} \right\},
$$
with $\alpha_* = {\alpha_0}/(1 + b_2/\log n)$, satisfy
$$
    \Pi_{W_n}(\Wcal_n^c)\le e^{-(c_2+4) n^{d/(2\alpha_0+d)}} = e^{-(c_2+4)n\varepsilon_n^2}.
$$
Set $\Pcal_n:=\{H\circ w, \ w\in\Wcal_n\}$. Then, by construction, $\Pi_n(\Pcal_n^c)\le \Pi_{W_n}(\Wcal_n^c) \le e^{-(c_2+4)n\varepsilon_n^2}$, showing that \eqref{Eq:Sieves} is indeed verified. Lastly, by Lemma 3.2 in \cite{vdVvZ08}, for all $w_1,w_2\in\Wcal_n$,
$$
    \| H\circ w_1 - H\circ w_2\|_1
    \lesssim \|w_1 - w_2\|_1
$$
and therefore
$$
	\log N(\varepsilon_n; \Pcal_n, \|\cdot\|_1)
    \lesssim \log N(\varepsilon_n; \Wcal_n, \|\cdot\|_1).
$$
Since, $n^{-\alpha_*/(2\alpha_*+d)}\lesssim \varepsilon_n$, cf.~\eqref{Eq:Equiv} below, and since, by construction, $\Wcal_n$ is a $b_1n^{-\alpha_*/(2\alpha_*+d)}$-enlargement in $L^1$-distance of the set $\{w:\|w\|_{B^{\alpha_*+d}_{1}}\le b_1 n^{d/(2\alpha_*+d)}\}$, the metric entropy of interest is upper bounded by a multiple of
\begin{align*}
	\log N\Big(\varepsilon_n; \left\{ w: \|w\|_{B^{\alpha_*+d}_{1}}
	\le b_1 n^\frac{d}{2\alpha_*+d} \right\}, \|\cdot\|_1\Big)
	&\lesssim 
	\left( \frac{b_1 n^\frac{d}{2\alpha_*+d} }{\varepsilon_n}\right)^{\frac{d}{\alpha_*+d}}\\
	&\lesssim 
	n^\frac{d}{2\alpha_*+d}
	\lesssim n\varepsilon_n^2,
\end{align*}
having used the metric entropy inequality in Theorem 4.3.36 in \cite{GN16} and again \eqref{Eq:Equiv}. This concludes the verification of \eqref{Eq:MetricEntropy} and the proof of Theorem \ref{Theo:LaplRates} in view of the equivalence between the total variation distance and the standard $L^1$-distance.
\qed

%
%
%

\subsection{Auxiliary results}

The next two auxiliary lemmas provide key quantitative properties of the information geometry of the hierarchical rescaled Besov-Laplace priors. They adapt to the present setting previous findings from \cite{giordano23besov}, which in turn were based on the investigations of \cite{lember2007universal} and \cite{vWvZ16}, where similar hyper-priors for the smoothness were considered.

\begin{lemma}\label{Lem:SmallBall}
Let $\Pi_{W_n}$ be the law of the hierarchical rescaled Besov-Laplace random element $W_n$ from \eqref{Eq:BaseLaplPrior} with smoothness hyper-prior $\alpha\sim\Sigma_n$ as in \eqref{Eq:Hyperprior}. Let $w_0\in B^{{\alpha_0}}_{1}([0,1]^d)$, for some ${\alpha_0}>d$. Then, for sufficiently large $B_1, B_2>0$,
$$
	\Pi_{W_n}\left(w:\|w - w_0\|_{2} \le B_1n^{-\frac{{\alpha_0}}{2{\alpha_0}+d}}
	\right)\ge e^{-B_2n^{d/(2{\alpha_0} + d)}}.
$$
\end{lemma}

\begin{proof}
For $\alpha>d$, let $\varepsilon_{\alpha,n} := n^{-\alpha/(2\alpha+d)}$ and let $\Pi_{W_{\alpha,n}}$ be the law of
\begin{equation}
\label{Eq:Wsn}
	W_{\alpha,n} 
	:= 
	\frac{W_{\alpha}}{n\varepsilon_{\alpha,n}^2},
	\qquad W_{\alpha}:=
	\sum_{l=1}^\infty l^{-\left(\frac{\alpha}{d} - \frac{1}{2}\right)}
	W_{l}\psi_{l}.
\end{equation}
Then, by construction, 
\begin{align*}
    \Pi_{W_n}&\left(w:\|w - w_0\|_{2} \le B_1n^{-\frac{{\alpha_0}}{2{\alpha_0}+d}}\right)\\
    & = \int_d^{\log n} 
    \Pi_{W_{\alpha,n}}\left(w:\|w - w_0\|_{2} \le B_1 \varepsilon_{\alpha_0,n} \right) \sigma_n(\alpha) d\alpha \\
    & \geq \int_{\alpha_0}^{\alpha_0 + 1/\log n} 
    \Pi_{W_{\alpha,n}}\left(w:\|w - w_0\|_{2} \le B_1 \varepsilon_{\alpha_0,n} \right) \sigma_n(\alpha) d\alpha.
\end{align*}
For a truncation level $L_n\in\N$ to be chosen below, the wavelet projection of $P_{L_n}w_0$ of $w_0\in B^{{\alpha_0}}_{1}([0,1]^d)$ satisfies
\begin{align*}
	\|w_0 - P_{L_n} w_0\|_2
	&=
	\sum_{l>L_n} l^{-\left(\frac{\alpha_0}{d} - \frac{1}{2}\right)} l^{\left(\frac{\alpha_0}{d} - \frac{1}{2}\right)}
	| w_{0,l} |
	\le
	L_n^{-\left(\frac{\alpha_0}{d} - \frac{1}{2}\right)}\|w_0\|_{B^{\alpha_0}_{1}}.
\end{align*}
Thus, taking $L_n \simeq n^{\alpha_0/[(2\alpha_0+d)(\alpha_0/d - 1/2)]}$ (note that this is greater than the usual order $n^{1/(2\alpha_0+d)}$), we have
$$
	\|w_0 - P_{L_n}w_0\|_2
	 \lesssim n^{-\frac{\alpha_0}{2\alpha_0+d}}
	 = \varepsilon_{\alpha_0,n}.
$$
Furthermore, for all $\alpha\in[\alpha_0,\alpha_0+1/\log n]$, it also holds 
\begin{align*}
	\|P_{L_n}w_0\|_{B^{\alpha}_{1}}
	&=\sum_{l\le L_n}
    l^{\left(\frac{\alpha}{d} - \frac{\alpha_0}{d}\right)} 
    l^{\left(\frac{\alpha_0}{d} - \frac{1}{2}\right)} 
	| \langle w_0,\psi_{l}
	\rangle_2 | \\
	&\le
		L_n^{\frac{\left(\alpha - \alpha_0 \right)}{d} } \|w_0\|_{B^{\alpha_0}_{1}}
	\lesssim
		n^{\frac{\alpha_0 \left(\alpha_0 + \log^{-1}n - \alpha_0 \right)} { (2\alpha_0+d)(\alpha_0 - d/2) }} = e^{\frac{\alpha_0}{(2\alpha_0+d)(\alpha_0-d/2)\log n}\log n} 	\lesssim 1.
\end{align*}
Hence, by the triangle inequality, for $B_1>0$ large enough and some $c_1>0$,
\begin{align*}
	\Pi_{W_{\alpha,n}}(w:\|w - w_0\|_2 \le B_1\varepsilon_{\alpha_0,n})
	\ge \Pi_{W_{\alpha,n}}(w:\|w - P_{L_n}w_0\|_2 \le c_1\varepsilon_{\alpha_0,n}).
\end{align*}
On the unit cube, we can further lower bound this quantity by,
$$ 
    \Pi_{W_{\alpha,n}}(w:\|w - P_{L_n}w_0\|_\infty \le c_1\varepsilon_{\alpha_0,n}), 
$$
which, by the decentering-inequality (32) in \cite{giordano23besov} for the rescaled Besov-Laplace random element $W_{\alpha,n}$, is greater than
\begin{align*}
	e^{-\|P_{L_n}w_0\|_{B^{\alpha}_{1}} n\varepsilon_{\alpha,n}^2 }
	 &\Pi_{W_{\alpha,n}}\left(w:\|w \|_\infty \le c_1\varepsilon_{\alpha_0,n}\right)\\
	 &\ge
	 e^{-c_2n\varepsilon_{\alpha_0,n}^2 }
	 \Pi_{W_\alpha}\left(w:\|w\|_\infty \le c_1 n \varepsilon_{\alpha_0,n}\varepsilon^2_{\alpha,n}\right).
\end{align*}
By the centered small ball inequality (34) in \cite{giordano23besov} (noting that $W_\alpha$ coincides with $W$ there with the choice $t=\alpha-d>0$),
\begin{align*}
	  \Pi_{W_\alpha}\left(w:\|w\|_\infty \le c_1 n \varepsilon_{\alpha_0,n}\varepsilon^2_{\alpha,n}\right)
	 \ge
		e^{-(c_3 (\alpha - d) + c_4) \left(c_1 n \varepsilon_{\alpha_0,n}\varepsilon^2_{\alpha,n}\right)^{-d/(\alpha-d)}} 
	\ge
		e^{-c_5n\varepsilon_{\alpha_0,n}^2},
\end{align*}
for $c_3,c_4,c_5>0$. Here, we used the fact that
\begin{align*}
	\left(n\varepsilon_{\alpha_0,n}\varepsilon^2_{\alpha,n}\right)^{-\frac{d}{\alpha-d}}
	&= \left(n \ n^{-\frac{\alpha_0}{2\alpha_0+d}} n^{\frac{2 \alpha}{2\alpha+d}} \right)^{-\frac{d}{\alpha-d}}\\
    &=\left(n^\frac{4\alpha\alpha_0 + 2d\alpha_0 + 2d\alpha + d^2 - 2\alpha\alpha_0 - d\alpha_0 - 4\alpha\alpha_0 - 2d\alpha}
	{(2\alpha+d)(2\alpha_0+d)}\right)^{-\frac{d}{\alpha-d}}\\
	&=\left(n^\frac{ \alpha_0d  + d^2 - 2\alpha\alpha_0  }{(2\alpha+d)(2\alpha_0+d)}
	\right)^{-\frac{d}{\alpha-d}}
	=
	\left(n^\frac{d }{2\alpha_0+d}
	\right)^{\frac{2\alpha\alpha_0 - d\alpha_0 - d^2}{(\alpha-d)(2\alpha+d)}}
	\le n\varepsilon_{\alpha_0,n}^2
\end{align*}
which holds provided the exponent is smaller than $1$, since $ n\varepsilon_{\alpha_0,n}^2=n^{d/2\alpha_0+d}$. To see this, note that
\begin{align*}
	(\alpha-d)(2\alpha+d) - (2\alpha\alpha_0 - d\alpha_0 - d^2) & \ge 0 
\end{align*}
following from $\alpha \ge \alpha_0 \ge d$. Combining the previous bounds, we find that for all $\alpha\in[\alpha_0,\alpha_0+1/\log n]$, for sufficiently large $B_1>0$,
$$
	\Pi_{W_{\alpha,n}}(w:\|w - w_0\|_\infty \le B_1\varepsilon_{\alpha_0,n})
	\ge e^{-c_6n\varepsilon_{\alpha_0,n}^2}.
$$
As a result, 
\begin{align*}
	\int_{\alpha_0}^{\alpha_0 + 1/\log n} &
    \Pi_{W_{\alpha,n}}\left(w:\|w - w_0\|_{2} \le B_1 \varepsilon_{\alpha_0,n} \right) \sigma_n(\alpha) d\alpha  
    \\ &\ge
    \int_{\alpha_0}^{\alpha_0+\frac{1}{\log n}}
	e^{-c_6n\varepsilon_{\alpha_0,n}^2}\sigma_n(\alpha)d\alpha\\
	&=
	e^{-c_6n\varepsilon_{\alpha_0,n}^2}
	\int_{\alpha_0}^{\alpha_0+\frac{1}{\log n}}
	\frac{e^{-n\varepsilon^2_{\alpha,n}}}{\zeta_n}d\alpha \\
    & \ge
        \frac{1}{\log n} \times \frac{e^{-n\varepsilon^2_{\alpha_0,n}}}{\zeta_n} 
        \gtrsim 
		(\log n)^{-2}e^{-n\varepsilon^2_{\alpha_0,n}}
	\ge
		e^{-c_7n\varepsilon^2_{\alpha_0,n}},
\end{align*}
for some $c_7>0$, where we used that $\sigma_n(\alpha)$ is increasing in $\alpha$, that the length of the integration interval is $1/\log n$ and that the normalization constant of $\sigma_n$ satisfies $\zeta_n\simeq \log n$. The claim then follows taking $B_2 = c_6 + c_7>0$.
\end{proof}

\begin{lemma}\label{Lem:Sieves}
Let $\Pi_{W_n}$ be the law of the hierarchical rescaled Besov-Laplace random element $W_n$ from \eqref{Eq:BaseLaplPrior} with smoothness hyper-prior $\alpha\sim\Sigma_n$ as in \eqref{Eq:Hyperprior}. For fixed $\alpha_0>d$, and $A_1,A_2>0$, let $\alpha_* = \alpha_0/(1 + A_2/\log n)$ and define the set
\begin{equation}
\label{Eq:SievesHier}
	\Wcal_n = \left\{w = w^{(1)} +  w^{(2)} , \ \|w^{(1)}\|_ 1
	 \le A_1n^{-\frac{\alpha_*}{2\alpha_*+d}},\ \|  w^{(2)}\|_{B^{\alpha_*+d}_{1}}
	\le A_2 n^\frac{d}{2\alpha_*+d}\right\}.
\end{equation}
Then, for all $K>0$, there exist sufficiently large $A_1,A_2$, such that for $n\in\N$ large enough,
$$
	\Pi_{W_n}(\Wcal_n^c)\le e^{-Kn^{d/(2\alpha_0+d)}}.
$$
\end{lemma}

\begin{proof}

For all $A_1,A_2>0$, the probability of interest is equal to
\begin{equation}
    \int_d^{\alpha_*}\Pi_{W_{\alpha,n}}(\Wcal_n^c)\sigma_n(\alpha)d\alpha + 
	\int_{\alpha_*}^{\log n}\Pi_{W_{\alpha,n}}(\Wcal_n^c)\sigma_n(\alpha)d\alpha.
   \label{Eq:IntBound1}
\end{equation}
We start analyzing the first integral, which is smaller than 
$$
	\int_d^{\alpha_*}\sigma_n(\alpha)d\alpha\le \alpha_* \frac{e^{-n\varepsilon_{\alpha_*,n}^2}}{\zeta_n}
	\le e^{-c_1n\varepsilon_{\alpha_*,n}^2},
$$
having upper bounded the size of the integration interval by $\alpha_*$ and exploited the fact that the hyper-prior p.d.f.~$\sigma_n$ is increasing in $\alpha$ and has normalizing constant $\zeta_n\simeq \log n$. By comparing $n\varepsilon_{\alpha_*,n}^2$ and $n\varepsilon_{\alpha_0,n}^2$, we obtain, 
\begin{align*}
	\frac{n\varepsilon_{\alpha_*,n}^2}{n\varepsilon_{\alpha_0,n}^2}
    &= n^{\frac{dA_2+d\log n}{dA_2 + (2\alpha_0+d)\log n} - \frac{d}{2\alpha_0 + d}} \\
    & =n^\frac{2dA_2\alpha_0}{(2\alpha_0+d)^2\log n+dA_2(2\alpha_0 + d)}
	= e^{\frac{2dA_2\alpha_0\log n}{(2\alpha_0+d)^2\log n+dA_2(2\alpha_0 + d)}},
\end{align*}
which shows that, for $c_2:=e^{\frac{2dA_2\alpha_0}{(2\alpha_0+d)^2}}>1$,
\begin{equation}
\label{Eq:Equiv}
\begin{split}
	   n\varepsilon_{\alpha_*,n}^2 
       & = n\varepsilon_{\alpha_0,n}^2  e^{\frac{2dA_2\alpha_0\log n}{(2\alpha_0+d)^2\log n+dA_2(2\alpha_0 + d)}} \nonumber \\
      & = n\varepsilon_{\alpha_0,n}^2  e^{\frac{dA_2\alpha_0}{(2\alpha_0+d)^2}} e^{\frac{dA_2\alpha_0(\log n (2\alpha_0+d)-dA_2)}{(2\alpha_0+d)^2  (\log n (2\alpha_0+d)+dA_2)}}  \ge \sqrt{c_2} n\varepsilon_{\alpha_0,n}^2,   
\end{split}
\end{equation}
as well as
\begin{align*}
	n\varepsilon_{\alpha_0,n}^2  e^{\frac{2dA_2\alpha_0}{(2\alpha_0+d)^2}} e^{-\frac{2d^2A_2^2\alpha_0}{(2\alpha_0+d)^2  (\log n (2\alpha_0+d)+dA_2)}}  \le c_2 n\varepsilon_{\alpha_0,n}^2
\end{align*}
holding for all $n\in \N$ large enough. For any $K>0$, we can then take $A_2$ large enough so that 
$$
	\int_d^{\alpha_*}\Pi_{W_{\alpha,n}}(\Wcal_n^c)\sigma_n(\alpha)d\alpha
	\le e^{-c_1n\varepsilon_{\alpha_*,n}^2}
	\le e^{-c_1\sqrt{c_2}n\varepsilon_{\alpha_0,n}^2}
	\le e^{-(K+1)n\varepsilon_{\alpha_0,n}^2}.
$$

Concerning the second integral in \eqref{Eq:IntBound1}, we write
\begin{align*}
	&\Pi_{W_{\alpha,n}}(\Wcal_n)\nonumber\\
	&=
	 \Pi_{W_{\alpha}}\Big(w = w^{(1)} +  w^{(2)}:  \|w^{(1)}\|_ 1
	 \le A_1n\varepsilon_{\alpha_*,n}\varepsilon_{\alpha,n}^2,\ \|  w^{(2)}\|_{B^{\alpha_*+d}_{1}}
	\le A_1 n^2\varepsilon^2_{\alpha_*,n}\varepsilon_{\alpha,n}^2\Big).
\end{align*}
Letting
\begin{align*}
	\overline\Wcal_n & := \Big\{\overline w= \overline w^{(1)} 
	+ \overline w^{(2)} + \overline w^{(3)} :  
	 \|\overline w^{(1)}\|_ 1\le n\varepsilon_{\alpha_*,n}\varepsilon_{\alpha,n}^2
	, \ \|\overline w^{(2)}\|_{H^{\alpha-d/2}}\le \sqrt{\overline A_1 n\varepsilon_{\alpha_*,n}^2}, \\
	&\quad\ \|\overline w^{(3)}\|_{B^{\alpha}_{1}}\le \overline A_1 n\varepsilon_{\alpha_*,n}^2 \Big\},
\end{align*}
the two-level concentration inequality (33) in \cite{giordano23besov} implies, using again \eqref{Eq:Equiv}, for $c_3,c_4>0$,
\begin{align*}
	\Pi_{W_{\alpha}}(\overline\Wcal_n)
	&\ge
		1- \frac{1}{\Pi_{W_{\alpha}}\left(w: \|w\|_ 1\le n\varepsilon_{\alpha_*,n}
		\varepsilon_{\alpha,n}^2\right)}
		e^{-c_3\overline A_1 n\varepsilon_{\alpha_*,n}^2}\\
	&\ge
		1- \frac{1}{\Pi_{W_{\alpha}}\left(w: \|w\|_ 1\le n\varepsilon_{\alpha_*,n}
		\varepsilon_{\alpha,n}^2\right)}
		e^{-c_4\overline A_1 n\varepsilon_{\alpha_0,n}^2}.
\end{align*}
As $\|w\|_ 1\le \|w\|_ \infty$, noting that $\alpha\ge \alpha_* = \alpha_0\log n/(M+\log n)>d$ for all $n$ large enough since $\alpha_0>d$, by the centred small ball inequality (34) in \cite{giordano23besov}, we have for all $\alpha_*<\alpha\le \log n$,
\begin{align*}
	 \Pi_{W_{\alpha}}\left(w: \|w\|_ 1\le n\varepsilon_{\alpha_*,n}\varepsilon_{\alpha,n}^2\right)
	&\ge
		e^{-(c_5 (\alpha - d) + c_6)\left(n\varepsilon_{\alpha_*,n}\varepsilon_{\alpha,n}^2\right)^{-d/(\alpha-d)}} \\
	&\ge
		e^{-c_7\log n\left(n\varepsilon_{\alpha_*,n}\varepsilon_{\alpha,n}^2\right)^{-d/(\alpha-d)}}.
\end{align*}
Using again \eqref{Eq:Equiv} and the fact that
\begin{align*}
	\log n\left(n\varepsilon_{\alpha_*,n}\varepsilon_{\alpha,n}^2\right)^{-\frac{d}{\alpha-d}}
	&=\log n
	\left(n^\frac{-2\alpha \alpha_* - d\alpha_* + 2d\alpha_* + d^2}{(2\alpha_*+d)(2 \alpha+d)}\right)^{-\frac{d}{\alpha-d}}\\
	&=\log n\left(n\varepsilon_{\alpha_*,n}^2\right)^{\frac{2\alpha \alpha_* - d\alpha_* - d^2}{(2 \alpha+d)(\alpha-d)}}
	\le n\varepsilon_{\alpha_*,n}^2,
\end{align*}
as the last exponent is strictly smaller than one, by virtue of $\alpha \ge \alpha_*$, 
we obtain
$$
	 \Pi_{W_{\alpha}}\left(w: \|w\|_ 1\le n\varepsilon_{\alpha_*,n}\varepsilon_{\alpha,n}^2\right)
	 \ge e^{-c_7 n\varepsilon_{\alpha_*,n}^2 }
	 \ge e^{-c_8 n\varepsilon_{\alpha_0,n}^2 }.
$$
For sufficiently large $\overline A_1>0$, it then follows that for all $\alpha\in [\alpha_*,\log n]$
\begin{align}
\label{Eq:ExpIneq1}
	\Pi_{W_{\alpha}}(\overline\Wcal_n)
	&\ge
		1-
		e^{-(c_4\overline A_1 - c_8) n\varepsilon_{\alpha_0,n}^2}
	\ge
		1-
		e^{-(K+1) n\varepsilon_{\alpha_0,n}^2}.
\end{align}

Now let $P_{L_n}\overline w^{(2)}$ be the wavelet projection of the term $\overline w^{(2)}$ in the definition of $\overline\Wcal_n$, at resolution $L_n\in \N$ such that $L_n \simeq n^\frac{d}{2 \alpha + d}$. Then,
\begin{align*}
	\|\overline w^{(2)} - P_{L_n}\overline w^{(2)}\|_ 1
	&\le L_n^{-\left(\frac{\alpha}{d} - \frac{1}{2}\right)} \|\overline w^{(2)}\|_{H^{\alpha-d/2}}\\
	&\lesssim n^{-\frac{\alpha -d/2}{2 \alpha+d}}\sqrt n \varepsilon_{\alpha_*,n}
	=n^\frac{d}{2 \alpha+d}\varepsilon_{\alpha_*,n}
	=n\varepsilon_{\alpha,n}^2\varepsilon_{\alpha_*,n}.
\end{align*}
Also, 
\begin{align*}
	\| P_{L_n}\overline w^{(2)}\|_{B^{\alpha}_{1}}
	&\lesssim \sqrt{L_n} \|\overline w^{(2)}\|_{H^{\alpha-d/2}}\\
	&\lesssim n^{\frac{d/2}{2 \alpha+d}}n^\frac{d/2}{2\alpha_*+d}
	=n^\frac{d\alpha_* + d^2/2 + d\alpha + d^2/2}{(2 \alpha+d)(2\alpha_*+d)}
	\le n^{\frac{d(d + \alpha_* + \alpha)}{(2 \alpha+d)(2\alpha_*+d)}} \le n\varepsilon_{\alpha_*,n}^2
\end{align*}
since $d + \alpha + \alpha_* < 2\alpha + d$ whenever $\alpha\ge \alpha_*$. Thus, setting $\widetilde w^{(1)} := \overline w^{(1)} + (\overline w^{(2)} - P_{L_n}\overline w^{(2)})$ and $\widetilde w^{(2)} := \overline w^{(3)} +  P_{L_n}\overline w^{(2)}$ yields that for all $\alpha\in [\alpha_*,\log n]$ and all $n$ and $\widetilde A_1$ large enough,
\begin{align*}
	\overline\Wcal_n 
	& 
	\subseteq 
	\widetilde\Wcal_n := 
	\{ \widetilde w = \widetilde w^{(1)} + \widetilde w^{(2)} : 
	 \|\widetilde w^{(1)}\|_ 1\le \widetilde A_1 n\varepsilon_{\alpha,n}^2\varepsilon_{\alpha_*,n}
	,\  \|\widetilde w^{(2)}\|_{B^{\alpha}_{1}}\le \widetilde A_1 n\varepsilon_{\alpha_*,n}^2 \}.
\end{align*}
In view of \eqref{Eq:ExpIneq1}, 
\begin{align}
\label{Eq:ExpIneq2}
	\Pi_{W_{\alpha}}(\widetilde \Wcal_n)
	\ge
	1- e^{-(K+1)n\varepsilon_{\alpha_0,n}^2}.
\end{align}
We conclude showing that, choosing sufficiently large $A_1>0$,
\begin{align}
\label{Eq:DesiredIncl}
	\widetilde \Wcal_n \subseteq \{w = w^{(1)} + w^{(2)}  :  \|w^{(1)}\|_ 1
	 \le A_1n\varepsilon_{\alpha,n}^2\varepsilon_{\alpha_*,n},\ \| w^{(2)}\|_{B^{\alpha_*+d}_{1}}
	\le A_1n^2\varepsilon_{\alpha,n}^2\varepsilon_{\alpha_*,n}^2\}
\end{align}
for all $\alpha\in [\alpha_*,\log n]$ and all $n\in \N$ large enough. To this aim, we start with the high-regularity case, $\alpha\in[ \alpha_* + d,\log n]$. Then 
$$
	\| \widetilde w^{(2)}\|_{B^{\alpha_*+d}_{1}}
	\le \| \widetilde w^{(2)}\|_{B^{\alpha}_{1}}
	\le \widetilde R n\varepsilon_{\alpha_*,n}^2
	\le
	\widetilde R n^2\varepsilon_{\alpha,n}^2\varepsilon_{\alpha_*,n}^2
$$
since $n\varepsilon_{\alpha,n}^2\to\infty$. The inclusion \eqref{Eq:DesiredIncl} thus follows with $w^{(1)} = \widetilde w^{(1)},\ w^{(2)} = \widetilde w^{(2)}$, and $R=\widetilde R$. For the remaining range $\alpha\in[\alpha_*, \alpha_* + d)$, we consider the wavelet projection $P_{L_n}\widetilde w^{(2)}$ of $\widetilde w^{(2)}$ at resolution level $L_n \simeq n^\frac{d^2}{(2 \alpha + d)(\alpha_*+d-\alpha)}$. Then,
\begin{align*}
	\|P_{L_n}\widetilde w^{(2)}\|_{B^{\alpha_*+d}_{1}}
	\le
	L_n^{\frac{\left(\alpha_*+d - \alpha \right)}{d} }\|\widetilde w^{(2)}\|_{B^{\alpha}_{1}}
	\lesssim 
	n^\frac{d}{(2 \alpha + d)} n\varepsilon_{\alpha_*,n}^2  
	=
	n^2\varepsilon_{\alpha,n}^2\varepsilon_{\alpha_*,n}^2,
\end{align*}
and, using the continuous embedding of $B^0_{1}$ into $L^1$ (e.g., eq.~(21), p.169 in \cite{ST87}),
\begin{align*}
	\|\widetilde w^{(2)} - P_{L_n}\widetilde w^{(2)}\|_ 1
	&\lesssim\|\widetilde w^{(2)} - P_{L_n}\widetilde w^{(2)}\|_{B^0_{1}}\\
	&\le L_n^{-\frac{\alpha}{d}} 
    \|\widetilde w^{(2)}\|_{B^{\alpha}_{1}}\\
	&\lesssim
	n^{-\frac{d\alpha}{(2 \alpha + d)(\alpha_*+d-\alpha)}}
	n^\frac{d}{2\alpha_*+d}
	= n^\frac{d^3 + \alpha_*d^2 - 2\alpha^2 d}{(2 \alpha + d)(2\alpha_*+d)(\alpha_*+d-\alpha)}.
\end{align*}
The inclusion \eqref{Eq:DesiredIncl} follows upon showing that the right hand side is smaller than 
$$
	n\varepsilon_{\alpha,n}^2\varepsilon_{\alpha_*,n} 
	= n^{\frac{d}{2 \alpha+d}}n^{-\frac{\alpha_*}{2\alpha_*+d}}
	= n^\frac{d^2 + \alpha_* d -2\alpha \alpha_* }{(2\alpha_*+d)(2 \alpha+d)}
	=n^\frac{d^3 + \alpha_* d^2  - \alpha^2 d - \alpha \alpha_* d  + 2\alpha^2 \alpha_* - 2 \alpha \alpha_*^2}
	{(2 \alpha + d)(2\alpha_*+d)(\alpha_*+d-\alpha)}.
$$
Indeed, the difference between the numerators of the exponents equals
\begin{align*}
	\Delta(\alpha)
	&= d^3 + \alpha_*d^2 - 2\alpha^2 d 
	- d^3 - \alpha_* d^2  + \alpha^2 d + \alpha \alpha_* d  - 2\alpha^2 \alpha_* + 2 \alpha \alpha_*^2\\
	&= -(2\alpha_* + d)\alpha^2
	+ (2\alpha_*^2 + d \alpha_*)  \alpha,
\end{align*}
which, as a function of $\alpha$, is a downward-pointing parabola with maximum attained at
$$
	\alpha_v := \frac{2\alpha_*^2 + d\alpha_* }{4\alpha_*+d} < \alpha_*.
$$
Since $\Delta(\alpha)$ is decreasing for $\alpha>\alpha_v$, for all $\alpha\in[\alpha_*,\alpha_*+d]$, 
\begin{align*}
	\Delta(\alpha)&\le
	\Delta(\alpha_*)\\
	&=
	-(2\alpha_* + d)\alpha_*^2
	+ (2\alpha_*^2 + d \alpha_*)  \alpha_* = 0.
\end{align*}
This shows as required that $\|\widetilde w^{(2)} - P_{L_n}\widetilde w^{(2)}\|_ 1\lesssim n\varepsilon_{\alpha,n}^2 \varepsilon_{\alpha_*,n}$, so that taking $w^{(1)} := \widetilde w^{(1)} + (\widetilde w^{(2)} - P_{L_n}\widetilde w^{(2)}) $ and $w^{(2)} := P_{L_n}\widetilde w^{(2)}$, the desired inclusion \eqref{Eq:DesiredIncl} follows for large enough $A_1>0$. By \eqref{Eq:ExpIneq2}, we then conclude
\begin{align*}
	\Pi_{W_{\alpha,n}}(\Wcal_n)
	&\ge
	\Pi_{W_{\alpha}}(\widetilde \Wcal_n)
	\ge
		1 - e^{-(K+1)n\varepsilon_{\alpha_0,n}^2}.
\end{align*}
Finally, combined with \eqref{Eq:IntBound1}, this yield
\begin{align*}
	\Pi_{W_n}(\Wcal_n^c)
	&\le e^{-(K+1)n\varepsilon_{\alpha_0,n}^2} + 
	\int_{\alpha_*}^{\log n}e^{-(K+1)n\varepsilon_{\alpha_0,n}^2}\sigma_n(\alpha)d\alpha\\
	&\le 2e^{-(K+1)n\varepsilon_{\alpha_0,n}^2}
	\le e^{-Kn\varepsilon_{\alpha_0,n}^2}.
\end{align*}
\end{proof}

%






\end{document}